\documentclass[reqno]{amsart}
\usepackage[utf8]{inputenc}
\usepackage{hyperref}
\usepackage{amsthm,amsmath,amssymb}
\usepackage{mathrsfs}
\usepackage{amsfonts}
\usepackage{tikz}
\usepackage{tkz-euclide}
\usepackage{amsthm}
\usepackage{amsmath,amscd}
\usepackage[all]{xy} 
\usepackage{graphicx}

\title[Combinatorial Rigidity]{On the Combinatorial Rigidity for Polynomials with Attracting Cycles}

\author{Yueyang Wang}
\thanks{Supported by the Fundamental Research Funds for the Central Universities}
\address{School of Mathematics, Shanghai University of Finance and Economics, Shanghai, 200433, China}
\email{yueyangwang@icloud.com}

\date{\today}

\newtheorem{thm}{Theorem}[section]
\newtheorem{cor}[thm]{Corollary}
\newtheorem{prop}[thm]{Proposition}
\newtheorem{lem}[thm]{Lemma}
\newtheorem{fact}[thm]{Fact}

\theoremstyle{remark}
\newtheorem{clm}{Claim}

\theoremstyle{definition}
\newtheorem{defi}{Definition}[section]

\newtheorem{exam}{Example}[section]

\def\f{f}
\def\a{{\mathcal a}}

\def\RZ{\mathbb{S}}
\def\Q{\mathbb{Q}}
\def\QZ{\Q/\Z}

\def\P{{\mathcal{P}_{d}}}

\def\C{\widehat{\mathbb C}}
\def\D{\mathbb{D}}

\def\L{\lambda}
\def\F{{\rm Fiber}}
\def\R{\mathbb{R}}
\def\C{\mathbb{C}}
\def\Z{\mathbb{Z}}

\def\AA{\mathcal{A}}

\def\b{{\mathbf{a}_2}}
\def\f{\mathbf{f}}
\def\comb{\mathbf{Comb}}
\def\qc{\mathbf{QC}}
\def\ca{\mathcal{C}(\AA)}

\def\CC{{\widehat{\mathbb{C}}}}
\def\a{{\bf a}}

\def\LL{\Delta}
\def\F{\mathcal{F}_d}

\def\RRR{R^0}

\def\I{\mathcal{I}}
\def\crit{\mathrm{Crit}}

\begin{document}

\maketitle

\begin{abstract}
    We show that every polynomial of degree $d \geq 2$ in the connectedness locus with an attracting cycle which attracts at least two critical points and no indifferent cycles is not combinatorially rigid.
    In particular, we prove that a hyperbolic polynomial with connected Julia set is combinatorially rigid if and only if it is of the ``disjoint type''.
\end{abstract}

\section{Introduction}


\subsection{Backgrounds}
We study the dynamical system of the iterations of polynomial maps.
Let $\P$ be the space of \emph{monic} and \emph{centered} polynomials in the form
$$
    f(z)=z^d+a_{d-2}z^{d-2}+\cdots+a_1z+a_0
$$
of degree $d \geq 2$ where $(a_0,a_1,\dots,a_{d-2}) \in \C^{d-1}$.

A useful tool in studying polynomial dynamics is the theory of external rays.
Let $\mathcal{C}_d$ denote the \emph{connectedness locus} of $\P$ which consists of polynomials in $\P$ with connected Julia sets.
For every $f \in \P$, there exists a canonical B\"ottcher map $\varphi_f$  which is a conformal map from the basin of $\infty$ to $\C\setminus \overline{\D}$ fixing $\infty$ and satisfies the properties that
$$
    \lim_{z\to \infty}\frac{\varphi_f(z)}{z} =1, \qquad \varphi_f(f(z))=(\varphi_f(z))^d.
$$
For every $\theta \in \RZ:=\mathbb{R}/\mathbb{Z}$, the \emph{external ray} with angle $\theta$ is defined to be
$$
    R_f(\theta):=\varphi_f^{-1}((1,+\infty)e^{2\pi i \theta}).
$$
We say that the external ray $R_f(\theta)$ \emph{lands} at $z$ if $\varphi_f^{-1}(re^{2\pi i \theta}) \to z$ as $r \to 1$.
A well-known fact is that if $f \in \mathcal{C}_d$, then every external ray with rational angles lands at a parabolic or repelling preperiodic point in the Julia set.
Conversely, every parabolic and repelling preperiodic point is the landing point of at least one external ray with a rational angle. 

We are interested in the combinatorics which encodes the landing properties of external rays.
McMullen \cite{McMullen1995TheCO} introduces following definition.

\begin{defi}[rational lamination]
     For $f \in \mathcal{C}_d$, the \emph{rational lamination} $\L_\Q(f)$ of $f$ is the equivalence relation defined on $(\Q/\Z)^2$ such that $(\theta,\theta') \in \L_\Q(f)$ if and only if the external rays $R_f(\theta)$ and $R_f(\theta')$ land at a common point.
\end{defi}

The rational lamination and related topics are intensively studied by Kiwi \cite{kiwi2001rational,wanderingorbitportrait,KIWI2004207,CombinatorialContinuity}, and Inou-Kiwi \cite{inou2012combinatorics,inou2009combinatorics}.

Rigidity is an interesting phenomenon in complex dynamics.
According to McMullen \cite{McMullen1995TheCO}, we have the following definition.

\begin{defi}[combinatorial rigidity]
    A polynomial $f \in \mathcal{C}_d$ without indifferent cycles is \emph{combinatorially rigid}, if for any $g \in \mathcal{C}_d$ with no indifferent cycles and $\L_\Q(f)=\L_\Q(g)$, the B\"ottcher composition $\varphi_{g}^{-1}\circ\varphi_f$ can be extended to a quasiconformal map $\phi$ on $\C$.
    In particular, $\phi$ forms a quasiconformal conjugation between $f$ and $g$ on the Julia sets.
\end{defi}

    It is conjectured that every polynomial $f \in \mathcal{C}_d$ with no indifferent cycles is combinatorially rigid.
    This is known as the \emph{Combinatorial Rigidity Conjecture}.
    For $d=2$, the Combinatorial Rigidity Conjecture implies the local connectivity conjecture of the Mandelbrot set (MLC), hence implies the Density of Hyperbolicity Conjecture for quadratic polynomials.
    The combinatorical rigidity is known for many cases, see \cite{hubbard1992local,avila2009combinatorial,kahn2006priori,kahn2008priori,kahn2009priori,cheraghi2010combinatorial}.

    For $d=3$, Henrikson \cite{Henriksen2003TheCR} find a counterexample to show that the Combinatorial Rigidity Conjecture fails in degree $3$.
    Recently, Kiwi-Wang \cite{CombinatorialContinuity} construct a unicritical counterexample in degree $3$.

    \subsection{Main Results}
    
    In this article, we study the combinatorial rigidity of polynomials with bounded Fatou cycles.
    Since the definition of combinatorial rigidity excludes all indifferent cases, we only need to discuss the attracting case.
    The main result of the article is the following.

\begin{thm}[main result]\label{main}
   For $d\geq 3$, every polynomial $f \in \mathcal{C}_d$ with no indifferent cycle and possessing an attracting cycle which attracts at least two critical points counted with multiplicity is not combinatorially rigid.
\end{thm}

    The condition that the attracting cycle attracts at least two critical points cannot be removed.
    In fact, a polynomial whose every attracting cycle attracts a unique critical point might be combinatorially rigid, see the following Theorem \ref{hyperbolic}.

    Our result has a wide generality since it has no restriction for the behaviors of the Julia critical points.
    It provides infinitely many counterexamples for the Combinatorial Rigidity Conjecture.


    A polynomial $f \in \mathcal{C}_d$ is \emph{hyperbolic} if all of critical points are attracted by attracting cycles.
    It is well-known that every attracting cycle must attract at least one critical point.
    Thus, a hyperbolic polynomial $f \in \mathcal{C}_d$ can possesses at most $d-1$ attracting cycles in the filled-in Julia set.
    The extreme case is call the \emph{disjoint type}.
    Applying Theorem \ref{main}, we see that for a hyperbolic polynomial $f \in \mathcal{C}_d$, non-disjoint type is a sufficient condition of combinatorial non-rigidity.
    Our second result shows that this condition is also necessary.
    
\begin{thm}[hyperbolic case]\label{hyperbolic}
   For $d\geq 2$, a hyperbolic polynomial in $\mathcal{C}_d$ is combinatorially rigid if and only if it is of the disjoint type.
\end{thm}

    
    We illustrate our idea with the following example which is already observed several times in the literature \cite{faught1993local,kiwi2001rational,roesch2007hyperbolic}.

\begin{exam}\label{eg1}
    The cubic polynomial $f(z) = z^3$ is not combinatorially rigid.
\end{exam}

    In fact, we can split the critical point $0$ of multiplicity $2$ into a fixed critical point $0$ and a simple critical point $\omega$.
    Then we push $\omega$ to a point with infinite forward orbit on the boundary of the fixed attracting component.
    Thus, we obtain a limit map $g$ on the boundary of the hyperbolic component containing $f$ such that the other critical point $\omega(g)$ lies on the boundary of the attracting Fatou component of $0$ with infinite forward orbit.
    Moreover, both $f$ and $g$ have trivial rational lamination but they are not conjugated on the Julia sets.
    See Figure \ref{fig1}.

 \begin{figure}[h]\label{fig1}
  \begin{center}
   \vspace{2mm}
   \begin{minipage}{.48\linewidth}
    \includegraphics[width=\linewidth]{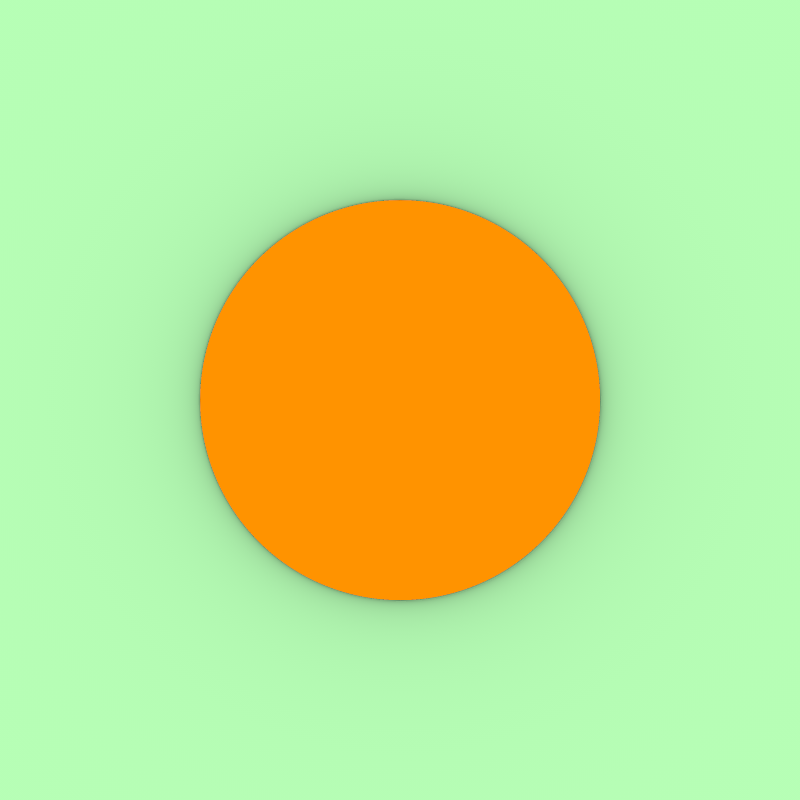}
    \caption{$K(f_0)$}
  \end{minipage}
  \hspace{2mm}
  \begin{minipage}{.48\linewidth}
    \includegraphics[width=\linewidth]{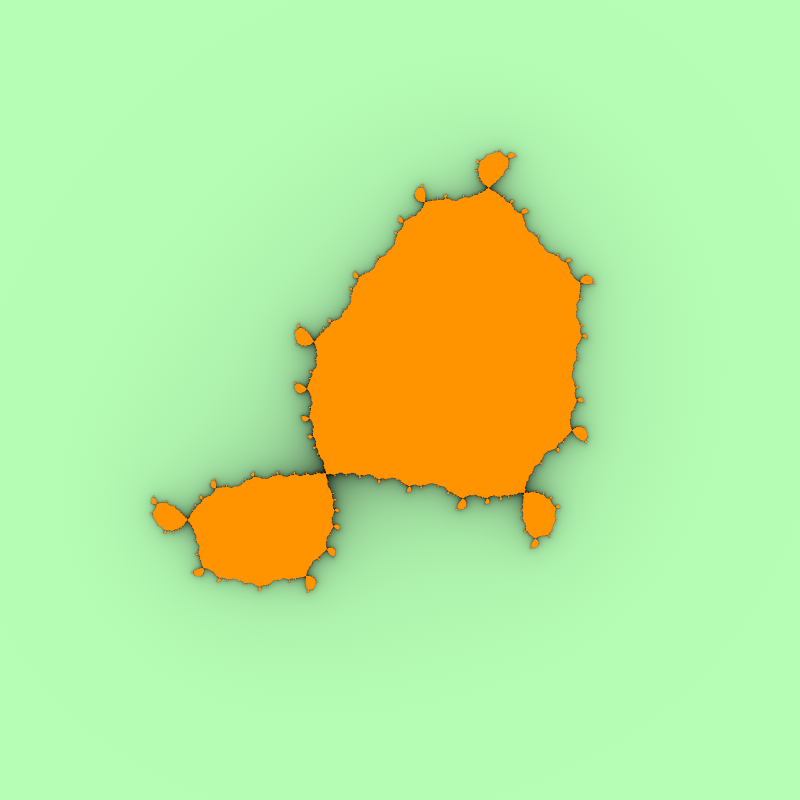}
    \caption{$K(f)$}
  \end{minipage}
 \end{center}
\end{figure}


\subsection{Idea of the proof and structure of the paper}
To prove Theorem \ref{main}, we need to find a map $g$ with $\L_\Q(g)=\L_\Q(f)$ which is not conjugated to $f$ on the Julia set.
The strategy of the proof of Theorem \ref{main} mainly follows by the inspiration of Example \ref{eg1}.
We push one of the critical point $c$ in the attracting basin toward a suitable boundary point $x$.
This is the main topic in Section \ref{sec3}.
This deformation forms a curve $\mathcal{R}=\{ f_t \}$ in the parameter space.
Then every limit map of the parameter curve $\mathcal{R}$ should serve as the desired map $g$.
It is obvious to see that $f$ and $g$ cannot be conjugated on the Julia sets.
Hence to finish the proof of Theorem \ref{main}, it suffices to show that $\L_\Q(g)=\L_\Q(f)$.
This requires us to characterize the dynamic of $g$ which is the main purpose of Section \ref{sec4}.

The main observation is that the dynamics of $g$ should inherit those of $f$ except the critical orbit of $c$.
The key difficulty is that in the general case, $f_t$'s might have Julia critical points with various kinds of orbits which is hard to be preserved to the limit $g$.
To solve this, we use the sectors, quadrilaterals, and puzzles and their perturbations to precisely control the positions of critical points of $g$.
There are two ingredients which are fundamental in our proof:
\begin{itemize}
    \item Roesch-Yin's theory of limb structure with respect to the attracting components (summarized in Theorem \ref{roeschyin});
    \item Douady-Hubbard's theory of polynomial-like maps (discussed and slightly generalized in Subsection \ref{subsec24}).
\end{itemize}

In Section \ref{sec5}, we present the proof of Theorem \ref{hyperbolic}.



\section{Preliminary}

\subsection{Notations}
Let $\mathbb{C}$ denote the complex plane.
The \emph{Riemann's sphere} $\CC:=\mathbb{C} \cup \{\infty\}$ is the one point compactification of $\mathbb{C}$. 
Let $\mathbb{D}:=\{z \in \C: |z|<1 \}$.

Let $\pi: \mathbb{R} \to \RZ:=\mathbb{R}/\mathbb{Z}$ be the natural projection map.
$\theta_{1},\theta_{2}, \cdots \theta_{n} \in \mathbb{R}/\mathbb{Z}$ are said to be in \emph{positive cyclic order} if there exist $\tilde{\theta}_{1}, \tilde{\theta}_{2}, \cdots, \tilde{\theta}_{n} \in \mathbb{R}$ such that $\pi(\tilde{\theta}_{k})=\theta_{k}$ for $1 \leq k \leq n$ and
\begin{equation*}
    \tilde{\theta}_{1}<\tilde{\theta}_{2}< \cdots <\tilde{\theta}_{n}<\tilde{\theta}_{1}+1.
\end{equation*}
For any $\theta_{1}, \theta_{2} \in \mathbb{R}/\mathbb{Z}$, the \emph{positive open interval} is defined to be
$$
    (\theta_{1},\theta_{2})_{+}:=\{\theta \in \RZ: \theta_{1},\theta,\theta_{2}~\mbox{are in positive cyclic order}\}.
$$
The \emph{length} of the interval $(\theta_{1},\theta_{2})_{+}$ is defined to be
$
    \LL (\theta_{1},\theta_{2})_{+}:=\tilde{\theta}_{2}-\tilde{\theta}_{1},
$
where $\tilde{\theta}_{1}, \tilde{\theta}_{2}$ are representatives of $\theta_{1},\theta_{2}$ such that $\tilde{\theta}_{1}<\tilde{\theta}_{2}<\tilde{\theta}_{1}+1$.
Similarly, one can define other kinds to intervals on $\RZ$. 
The length of these intervals are all defined to be $ \LL (\theta_{1},\theta_{2})_{+}$.
For $d \geq 2$, let $\tau_d: \RZ \to \RZ$ be the induced map of the map $\theta \mapsto d \theta : \mathbb{R} \to \mathbb{R}$ on the quotient space $\RZ$.

Let $\mathcal{A}$ be a complex manifold.
Given $\a_{0} \in \AA$ and a set $X \subset \CC$, a map $h: \AA \times X \to \CC$ is called a \emph{holomorphic motion} ({based at $\a_{0}$}) if the following holds.
\begin{enumerate}
    \item Fix $x \in X$, the map $h(\cdot,x): \a \mapsto h(\a,x)$ is a holomorphic function.
    \item Fix $\a \in \AA$, the map $h_\a: x \mapsto h(\a,x)$ is an injection. 
    \item The map $h_{\a_0}: X \to X$ is the identity map. 
\end{enumerate}
According to the well-known $\lambda$-Lemma \cite{mane1983dynamics}, any holomorphic motion $h: \AA \times X \to \CC$ can be extended to a holomorphic motion $h: \AA \times \overline{X} \to \CC$.

\subsection{External rays}
Let $\mathcal{F}=\{ f_\a\}_{\a \in \AA}$ be a family of polynomials of degree $d \geq 2$.
For $\a \in \mathcal{F}$,
$\infty$ is a fixed critical point of $f_\a$ with multiplicity $d$.
Let 
$$
    K_\a:=\{z \in \mathbb{C} : \{ f_\a^n(z) \} \mbox{ is bounded} \}.
$$
denote the filled-in Julia set, and $J_\a:=\partial K_\a$ be the Julia set.
Then $\Omega_\a:= \C \setminus K_\a$ is the \emph{attracting basin} of $\infty$.
There exists a B\"ottcher map $\varphi_\a$ which is defined near $\infty$ such that 
\begin{equation}\label{boteq3}
    \varphi_\a(f_\a(z))=(\varphi_\a(z))^d,\qquad \lim_{z \to \infty}\frac{\varphi_\a(z)}{z}=1. 
\end{equation}
In general, $\varphi_\a$ cannot be continuously extended to the whole $\Omega_\a$ while its modulus can be extended to $\Omega_\a$.
It is the \emph{exponential Green function} $|\varphi_\a| : \Omega_\a \to [0,+\infty)$ of $f_\a$
Fix $\a\in \AA$, the map $\varphi_\a$ is called the \emph{B\"ottcher map} of $f_\a$ with respect to $\infty$.  
Denote $\psi_{\a}$ to be the inverse of $\varphi_{\a}$.
By \eqref{boteq3}, we have 
\begin{equation}\label{boteq4}
   f_\a \circ \psi_{\a} (w)=\psi_{\a} (w^d),
\end{equation}
According to \cite{douady1984etude}, there exists a unique $s_{\a}(\theta) \in [1, {s_{\a}}]$ such that $\psi_\a$ admits a canonical (which means that \eqref{boteq4} still holds) smooth extension to $(s_{\a}(\theta),+\infty)e^{2\pi i \theta}$ without hitting any iterative preimages of critical points.
If $s_{\a}(\theta)=1$, we define the \emph{external ray}  with angle $\theta$ is defined to be
$$
    R_{\a}(\theta):=\psi_{\a}((1,+\infty)e^{2\pi i \theta}).
$$
If $\lim_{s \to 1} \psi_{\a}(se^{2\pi i \theta})=z$ exists, we say that the external ray $ R_{\a}(\theta)$ \emph{land}s at $\psi_{\a}(e^{2\pi i \theta}):=z$.
If $\theta \in \mathbb{Q}/\mathbb{Z}$ and $s_{\a}(\theta)=1$, then $R_{\a}(\theta)$ lands at an eventually repelling or parabolic periodic point.
If $K_\a$ is connected, then for any eventually repelling or parabolic periodic point $z$, there exist at least one, and at most finitely many $\theta \in \mathbb{Q}/\mathbb{Z}$ such that the external ray $R_\a(\theta)$ lands at $z$.

The following two well-known facts concerning the perturbation of external rays are well-known, see \cite{douady1984etude}.

\begin{lem}[holomorphic motion of external rays]\label{rayhm}
    Let $\mathcal{F}=\{ f_\a\}_{\a \in \mathcal{U}}$ be a holomorphic family of polynomials of degree $d \geq 2$.
    Suppose that for $\a \in \mathcal{U}$, we have $s_\a(\theta)=1$, then fix any $\a_0 \in \mathcal{U}$, the map $\xi : \mathcal{U} \times R_{\a_0}(\theta) \to \C$ given by $(\a, z)\mapsto \psi_\a\circ\varphi_{\a_0}(z)$ is a holomorphic motion.
     
\end{lem}

\begin{lem}[continuity of external rays landing on repelling orbit]\label{stablerepelling}
    Suppose that $R_{\a_0}(\theta)$ lands at a repelling periodic or preperiodic point such that the forward orbit of $\psi_{\a_0}(e^{2\pi i \theta})$ does not contain any critical points of $f_{\a_{0}}$.
    There exists a neighborhood $\mathcal{U}$ of $\a_{0}$ such that for any $\a \in \mathcal{U}$, $s_{\a}(\theta)=1$ and $z_\theta(\a)$ is also eventually repelling.
\end{lem}

We will also need Roesch-Yin's characterization of limbs structure with respect to the bounded Fatou components.

\begin{thm}[Roesch-Yin \cite{roesch2008boundary}]\label{roeschyin}
    Let $f_\a$ be a polynomial of degree $d \geq 2$ with connected Julia set, and $B_\a$ be a periodic attracting component of period $p$.
    Then $\partial B_\a$ is a Jordan curve.
    For each $x \in \partial B_\a$, there exists a compact connected set $L_x \subset K_\a \setminus B_\a$ such that $L_x \cap \overline{B_\a}=\{ x\}$ and
    $$K_\a=B_\a \cup \left(\bigsqcup_{x \in \partial B_\a} L_x \right).$$
    Here $L_x$ is called the limb of $B_\a$ attached at $x$, and $x$ is called the root of $L_x$.
    Moreover, the following holds.
   \begin{enumerate}
       \item If $L_x=\{ x\}$, then there exists a unique external ray which lands at $x$.
       \item If $L_x\neq \{ x \}$, then there exist two distinct external rays $R_\a(\theta^-)$, $R_\a(\theta^+)$ landing at $x$ such that they together with their common landing point separate $L_x$ away from $B_\a$.
       There exists $n$ such that $L_{x_n}$ contains a critical point of $f_\a^p$ where $x_n:=f_\a^n(x)$.
   \end{enumerate}
\end{thm}

The limb $L_x$ can be constructed as follows.
Suppose that $D:=\deg(f_\a^p|_{B_\a})$, 
then there exists a homeomorphic parameterization $\gamma: \RZ \to  \partial B_\a$ such that 
\begin{equation}\label{gammaeq}
    f_\a^p \circ \gamma(t) = \gamma(t^D), \qquad \forall t \in \RZ.
\end{equation}
Let $t_x:=\gamma^{-1}(x) \in \RZ$, then for any $(t^-,t^+)_+ \ni t_x$ with $t^\pm \in \QZ$, $\gamma(t^\pm)$ are parabolic or repelling preperiodic points on $\partial B_\a$.
There exist $\theta^\pm$ such that the external ray $R_\a(\theta^\pm)$ lands at $\gamma(t^\pm)$.
For $s'>1$ define $L_\a((t^-,t^+)_+,s')$ to be the closure of bounded complementary component of $\C \setminus (\psi_a(\{ z: |z|=s' \}\cup \gamma([t^-,t^+]_+)\cup R_\a(\theta^-)\cup R_\a(\theta^+))$.
Then we have
\begin{equation}\label{limbdef}
    L_x = \bigcap_{ (t^-,t^+)_+ \ni t_x ,s'>1 } L_\a((t^-,t^+)_+,s').
\end{equation}

\subsection{Polynomial-like maps}\label{subsec24}

Recall the definition of polynomial-like maps introduced by Douady and Hubbard \cite{douady1985dynamics}.
Let $U \Subset V$ be two simply-connected domains, and $f : U \to V$ be a holomorphic proper map of degree $d \geq 2$.
Then $\f:=(f,U,V)$ is called a \emph{polynomial-like map of degree $d$}.
Its filled-in Julia set is defined to be 
$$K(\f):=\{ z \in U : f^n(z) \in U ~\mbox{for all $n$}\}.$$
The Straightening Theorem asserts that $\f$ is \emph{hybrid-equivalent} to a polynomial $g$ of degree $d$, i.e. there exsits a quasiconformal conjugacy $\eta$ defined on $V$ such that $\eta \circ f = g \circ \eta$ and $\eta$ with $\frac{\partial \eta}{\partial \overline{z}}=0$ on $K$. 
Moreover, if $K$ is connected, then $g$ is unique up to an affine conjugation.
The map $\eta$ is called the \emph{straightening map} of $\f$, and $g$ is called the \emph{renormalized map} of $\f$.

We remark that the definition of polynomial-like maps and the Straightening Theorem trivially holds for degree $d=1$.
In fact, if $f: U \to V \Supset U$ is a proper map of degree $1$.
Then the filled-in Julia set $K(\f)$ is a repelling fixed point, and $g$ is a linear map $z \mapsto \lambda z$ with some $|\lambda|>1$.
The straightening map coincides with the Koenig's linearization coordinate.
Thus, we can talk about polynomial-like maps of any degree $d \geq 1$.

The definition of polynomial-like maps and the Straightening Theorem can be slightly generalized as follows.
Let $U_1\Subset V_1,U_2\Subset V_2$ be $4$ simply-connected domains, and $\f_1:=(f_1,U_1,V_1)$ is a polynomial-like map of degree $d_1$, $f_2: V_2 \to V_1$ and $f_2|_{U_2}: U_2 \to U_1$ are proper map of degree $d_2$.
Denote $\f_2:=(f_2,U_2,V_2)$.
We call $(\f_1,\f_2)$ a \emph{generalized polynomial-like map} of degree $(d_1,d_2)$.
We can also define it corresponding filled-in Julia set $K=(K(\f_1),K(\f_2))$ where $K(\f_2):=f_2^{-1}(K(\f_1))$.
Then we can generalize Douady-Hubbard Straightening Theorem as follows.

\begin{prop}[straightening theorem]\label{straightening}
    Let $(\f_1,\f_2)$ be a generalized polynomial-like map of degree $(d_1,d_2)$, then there exists two quasiconformal mappings $\eta_1$ define on $U_1$, $\eta_2$ defined on $U_2$, and two polynomials $g_1$ of degree $d_1$ and $g_2$ of degree $d_2$ such that
    $\eta_1 \circ f_i=g_i\circ\eta_i$ for $i=1,2$.
    We can view it as the following diagram.
\begin{equation*}
    \xymatrix{\ar @{} [dr]
  U_2 \ar[d]^{\eta_2} \ar[r]^{f_2}  &  U_1 \ar[d]^{\eta_1} \ar[r]^{f_1} & V_1 \ar[d]^{\eta_1} \\
  \C  \ar[r]^{g_2}  &   \C \ar[r]^{g_1}        & \C .     }
\end{equation*}    
    Moreover, if $K(\f_1)$ and $K(\f_2)$ are both connected, then $g_1$ and $g_2$ are both unique up to affine conjugacies. 
\end{prop}

\begin{proof}

    The uniqueness part follows from \cite[Theorem 3.53]{Branner_Fagella_2014}.
    In the following, we focus on the existence part.

    By restricting $f_1$ and $f_2$ to smaller domains, we may assume that $U_1,U_2,V_1,V_2$ all have smooth boundaries.
    Pick $r>1$ and conformal map $\phi: \C\setminus \overline{V_1} \to \{ z : |z|>r \}$.
    This map can be extended homeomorphically to the boundary.
    Define a quasi-regular map $F: \C  \to \C  $ such that $\phi\circ F(z)=(\phi(z))^{d_1}$ for $z \in \C\setminus V_1$ and  $F(z)=f_1(z)$ for $z \in U_1$.     
    Let $\sigma_1$ be the $F$-invariant complex structure such that $\sigma_1$ coincides with the standard complex structure on $\C \setminus V_1$.
    By the Measurable Riemann Mapping Theorem, there exists a  quasiconformal mapping $\eta_1 : \C \to \C$ with $\lim_{z\to \infty}\eta_1(z)=\infty$ such that $\eta_1^{*}(\sigma_0)=\sigma_1$.
    Thus, $g_1:=\eta_1\circ F\circ\eta_1^{-1}$ is a polynomial of degree $d_1$.
    Since $F$ coincides with $f_1$ on $U_1$, we have $\eta_1 \circ f_1=g_1\circ\eta_1$.    

    Define a conformal map $\psi : \C \setminus \overline{V_2} \to \{ z : |z|>\sqrt[d_2]{r} \}$ 
    which can be extended homeomorphically to the boundary.
    Define a quasi-regular map $F_2: \C  \to \C  $ such that $\phi\circ F_2(z)=(\psi(z))^{d_2}$ for $z \in \C\setminus V_2$ and $F_2(z)=f_2(z)$ for $z \in U_2$.
    Define another complex structure $\sigma_2:=F_2^*(\sigma_1)$. 
    By the Measurable Riemann Mapping Theorem, there exists a  quasiconformal mapping $\eta_2 : \C \to \C$ with $\lim_{z\to \infty}\eta_2(z)=\infty$ such that $\eta_2^{*}(\sigma_0)=\sigma_2$.    
    Let $g_2:=\eta_1\circ F_2\circ\eta_2^{-1}$, then $g_2^*(\sigma_0)=( \eta_1\circ F_2\circ\eta_2^{-1} )^*(\sigma_0)=(\eta_2^{-1})^*F_2^*\eta_1^*(\sigma_0)=\sigma_0$.  
    Hence $g_2$ is a polynomial of degree $d_2$.
    Since $F_2$ coincides with $f_2$ on $U_2$, we have $\eta_1 \circ f_2=g_2\circ\eta_2$.

\end{proof}

The map $(\eta_1,\eta_2)$ is called the \emph{the straightening map} of $(\f_1,\f_2)$, and $(g_1,g_2)$ is called the \emph{renormalized map} of $(\f_2,\f_2)$.

A more general version of generalized polynomial-like maps and the straightening theorem is discussed by Inou-Kiwi \cite{inou2012combinatorics}.
Proposition \ref{straightening} is in fact a special case of \cite[Theorem A]{inou2012combinatorics}.



\section{Perturbation of the Free Critical Orbits}\label{sec3}


We say two polynomial $f$ and $g$ in $\mathcal{C}_d$ are \emph{combinatorially equivalent} if $\L_\Q(f)=\L_\Q(g)$.
Denote $\comb(f)$ to be the combinatorial equivalence class of $f$.
We say that $f$ and $g$ in $\mathcal{C}_d$ are \emph{q.c. equivalent} if the B\"ottcher transition $\varphi_g^{-1}\circ\varphi_f$ extends to a quasiconformal map $\eta:\C \to \C$ such that $g\circ\eta=\eta \circ f$ on the Julia set of $f$.
Denote $\qc(f)$ to be the q.c. equivalence class of $f$.

It is easy to see that $\qc(f) \subset \comb(f)$.
Notice that if $f$ is combinatorially rigid if and only if $\qc(f)=\comb(f)$.

\subsection{Unique Fatou critical point with infinite orbit}\label{subsec31}
Let $f$ be a polynomial of degree $d\geq 3$ with some attracting periodic orbit $O$ which attracts at least two critical points counted with multiplicity.
By applying a translation conjugation, we may assume that $0 \in O$ and its attracting component contains a critical point.

\begin{lem}[unique Fatou critical point of infinite orbit]\label{centralize}
    There exists a polynomial $g \in \qc(f)$ such that every critical point in attracting basins has finite forward orbit except a unique critical point $\omega$ attracted to $0$ with infinite forward orbit.
\end{lem}

\begin{proof}

    The proof is done in two steps.
    First, we establish the following claim.

    \begin{clm}\label{centralize2}
        There exists a polynomial $h \in \qc(f)$ such that every critical point in attracting cycles has finite forward orbit.
    \end{clm}

    \begin{proof}[Proof of Claim \ref{centralize2}]
    We apply quasiconformal surgeries to every critical component to change the position of critical points.
    The construction is done component by component by the following rule:
    The periodic critical components will be modified before the strictly preperiodic ones.
    
    Suppose that $f$ has $n$ critical component containing critical points with infinite forward orbit.
    We will construct a finite sequence of maps $\{g_k\} \subset \qc(f)$ for $0\leq k \leq n$ such that $g_0=f$, and $g_k$ has $n-k$ critical attracting components containing critical points with infinite forward orbit.
    In every step $0\leq k \leq n-1$, we pick a critical attracting component $B_k$ containing critical points with infinite forward orbit.
    By induction and the order we pick, we may assume that $g_k(B_k)$ has a unique preperiodic point $b_k$.
    Take simply-connected domains $U_k\Subset V_k\Subset B_k$ such that $g_k: U_k \to g_k(U_k)$ and $g_k: V_k \to g_k(V_k)$ are both proper maps of degree $\delta_k=\deg(g_k|_{B_k})$.
    Pick Riemann mappings $\phi_k: U_k \to \D$, $\psi_k : g_k(U_k) \to \D$ such that $\psi(b_k)=0$.
    Define a quasiregular map $F_k: \C \to \C$ such that $F_k|_{\C \setminus V_k}=g_k$ and $g_k(z)=\psi_k^{-1}((\phi_k(z))^{\delta_k})$ for $z \in U_k$.
    Pulling back the standard complex structure $\sigma_0$ near $\infty$ and $b$, we get a $F_k$-invariant complex structure $\sigma_k$ with bounded dilatation.
    By the Measurable Riemann Mapping Theorem, there exists a quasiconformal map $\xi_k: \C \to \C$ such that $\xi_k^*(\sigma_0)=\sigma_k$ with $\xi_k(\infty)=\infty$, $\lim_{z\to \infty}\frac{\xi_k(z)}{z}=1$ and $\xi_k(0)=0$.
    It follows that $g_{k+1}:=\xi_k \circ F_k \circ \xi_k^{-1}$ is a polynomial of degree $d$ which fixes $0$.
    Since $g_{k+1}$ is conjugated to $F_k$ in $B'_k=\xi_k(B_k)$, there exists a unique critical point of $g_{k+1}$ in $B'_k$.
    Since $\sigma=\sigma_0$ and $F_k=g_k$ outside the basin of $B_k$, $\xi_k$ forms a conformal conjugation between $g_k$ and $g_{k+1}$ outside $B_k$, we see that $g_k$ has $n-k-1$ critical attracting components containing critical points with infinite forward orbit.
    Moreover, we have $g_{k+1} \in \qc(g_k)=\qc(f)$.
    By repeating the procedure $n$ times, we find the polynomial $g_n$ whose critical points in Fatou components have finite forward orbit.
    Set $h=g_n$, the proof is finished.
    \end{proof}


    By the assumption, there exists another critical point $\omega'$ such that $h^m(\omega')=0$ for some $m \geq 0$.
    To meet the requirement of $g$, we need to change the position of $\omega'$ to have infinite forward orbit.
    The construction is very similar to the construction before.
    We can define a quasi-regular map $F: \C \to \C$ agreed with $h$ outside a neighborhood of $\omega'$ such that the forward orbit of $\omega'$ under $F$ is infinite.
    $F$ determines a complex structure with bounded dilatation, then $F$ is quasiconformally conjugated to a polynomial $g$ such that $g\in \qc(h)=\qc(f)$ and $g$ has a unique Fatou critical point $\omega$ with infinite forward orbit.

\end{proof}

Since $g \in \qc(f)$, we replace $f$ by $g$ in the following discussion.
Hence we may assume that $f$ itself satisfies the Lemma \ref{centralize}.

\subsection{Space with marked critical points}\label{subsec32}

Let $\F$ denote the space of monic polynomials of degree $d$ with a critical point $0$.
Then we have $f \in \F$.
We parameterize the space $\F$ by their critical points.
Define a map $\rho: \C^{d-1} \to \F$ by assigning each  $\a=(c,c_2,\dots,c_{d-2},b) \in \C^{d-1}$ to the polynomial
$$
    f_\a(z):=d\int_{0}^{z}   \zeta(\zeta-c)(\zeta-c_2)\cdots(\zeta-c_{d-2})   \mathrm{d} \zeta+b
$$
in $\F$ whose critical points precisely are $0,c,c_2,\dots,c_{d-2}$. 
These critical points are holomorphic maps of $\a \in \C^{d-1}$ and will be denoted by $0,c(\a),c_2(\a),\dots,c_{d-2}(\a)$.

Since $\rho$ is surjective, we can lift the polynomial $f$ to some $\a_0 \in \C^{d-1}$ such that $f_{\a_0}=f$ and $c(\a_0)=\omega$ is the critical point which is attracted to $0$ with infinite forward orbit.
The dynamical property of $f=f_{\a_0}$ is concluded as follows.
We decompose $\{ 2,3,\dots,d-2 \}$ into the union of $I_F$ and $I_\infty$.
\begin{enumerate}
    \item The critical point $\omega=c(\a_0)$ is attracted to the Fatou compnent of $0$.
    \item  $0$ and $c_k(\a_0)$ for $k \in I_F$ are critical points with finite forward orbit.
    For every $k \in I_F$, there exist $n_k,m_k$ such that
    $$
        f_{\a_0}^{m_k}(c_k(\a_0)) = f_{\a_0}^{n_k}(c_k(\a_0)), \quad k \in I_F.
    $$
    There exists a minimal $p$ such that $f_{\a_0}^p(0)=0$.

    \item  $c_k(\a_0)$ for $k \in I_\infty$ are Julia critical points with infinite forward orbit.   
\end{enumerate}
Let $d_0$ be the local degree of $f_{\a_0}$ at $0$.

To push the critical point $c(\a_0)$, we should fix the orbit of other critical points.
Thus, the pushing process will be done in the a sub-manifold of $\F$.
It can be parameterized as follows.

\begin{defi}[sub-manifold]
    Let $\AA \subset \C^{d-1}$ to be the sub-manifold containing $\a$ such that $f_\a^p(0)=0$ and $f_\a^{m_k}(c_k(\a))=f_\a^{n_k}(c_k(\a))$ for $k \in I_F$.
\end{defi}

Let $\ca$ denote the connectedness locus of $\AA$ which consists of $\a \in \AA$ such that $K_\a$ is connected.
By the definition of $\AA$, we have $\a_0 \in \AA$.
For every $\a \in \AA$, denote the attracting component containing $0$ by $B_\a$.
Sine $0$ is super-attracting, there is a B\"ottcher map $\phi_\a$ near $0$ whose modulus $|\phi_\a|$ can be extended continuously to $B_\a$.
For $0<s<1$, define $B_\a(s):=|\phi_\a|^{-1}([0,s))$.
Let 
$$s_\a:=\min \{ |\phi_\a|(z) : z \in {\rm Crit}(f_\a^p) \cap B_\a  \}.$$
Then $\phi_\a$ can be conformally extended to $B_\a(s_\a)$.

Similar the discussion of the basin of $\infty$, for every $t \in \RZ$, there exists a $s_\a^0(t) \in (0,1)$ such that the inverse $\psi_\a^0$ of $\phi_\a$ can be cannonically extended to the arc $[0,s_\a^0(t)]e^{2\pi i t}$ until it hits an iterative preimage of a critical point.
If $s_\a^0(t)=1$, we call the curve $\psi_\a^0([0,1)e^{2\pi i t})$ the \emph{internal ray} of angle $t$, and denoted by $\RRR_\a(t)$.
If the limit $\lim_{s\to 1}\psi_\a^0(se^{2\pi i t})=z$ exists, we say that the internal ray $\RRR_\a(t)$ lands at $z$.
Similar the case of external rays, for every $t \in \Q/\Z$ with $s_\a^0(t)=1$, then the internal ray $\RRR_\a(t)$ lands at a repelling or parabolic preperiodic point.
We also have a corresponding result to Lemma \ref{stablerepelling} concerning the stability of the internal rays landing at repelling preperiodic points. 

\subsection{Sectors, quadrilaterals, limbs}\label{subsec33}
In this part, we introduce the sectors and quadrilaterals with respect to the the Fatou component of $0$.
Given $\a \in \ca$, we say $\mathcal{I}=( (t^-,t^+)_+, (\theta^-,\theta^+)_+ ) $ is a \emph{combinatorics} for  $\a $, if
\begin{itemize}
    \item[(C1)] $s_\a^0(t^-)=s_\a^0(t^+)=1$;
    \item[(C2)] the internal rays $\RRR_{\a}(t^\pm)$ lands at repelling preperiodic points $z^\pm \in \partial B_{\a}$ which is not in the forward orbit of any critical point;
    \item[(C3)] the external rays $R_{\a}(\theta^\pm)$ land at $z^\pm$.
\end{itemize}
We call $(t_-,t_+)_+$ the \emph{internal arc} of $\I$, and $(\theta_-,\theta_+)_+$ the \emph{external arc} of $\I$. 
We define the \emph{sector} $S_{\a}(\I)$ with respect to $\I$ to be the connected component of $\C \setminus (\overline{R_{\a}(\theta^-)} \cup \overline{R_{\a}(\theta^+)} \cup \overline{\RRR_{\a}(t^-)} \cup \overline{\RRR_{\a}(t^+)} )$ containing internal rays with angles in $(t^-,t^+)_+$.
By the condition (C2) and Lemma \ref{stablerepelling}, we have the following fact concerning the holomorphic motion of sectors.

\begin{fact}[holomorphic motion of sectors]\label{fact} 
    Let $S_\a(\I)$ be a sector, there exists a neighborhood $\mathcal{U}\subset \AA$ of $\a$ such that for $\a' \in \mathcal{U}$, $\I$ is also a combinatorics for $\a'$, i.e. $S_{\a'}(\I)$ is well-defined.
    Moreover, there exists a holomorphic motion $h:\mathcal{U} \times S_\a(\I) \to \C$, $(\a,z)\mapsto h_\a(z)$ such that $h_{\a'}(\partial S_\a(\I))=\partial S_{\a'}(\I)$.
\end{fact}

Given $s<1,s'>1$, define the \emph{quadrilateral} with respect to $\I$ to be 
$$ Q_{\a}(\I,s,s'):= S_{\a}(\I) \setminus (\psi_{\a}(\{z:|z| \geq s'\}) \cup \psi_{\a}^{0}(\{ z : |z| \leq s \}) ).$$
By Fact \ref{fact}, as long as $K_\a$ is connected and $0$ is the unique critical point in the attracting basin of $0$, $ \partial Q_{\a}(\I,s,s') $ also admits a holomorphic motion in a neighborhood of $\a$ in $\AA$.

 \begin{figure}[h]\label{fig1}
  \begin{center}
   \vspace{2mm}
   \begin{minipage}{.48\linewidth}
    \includegraphics[width=\linewidth]{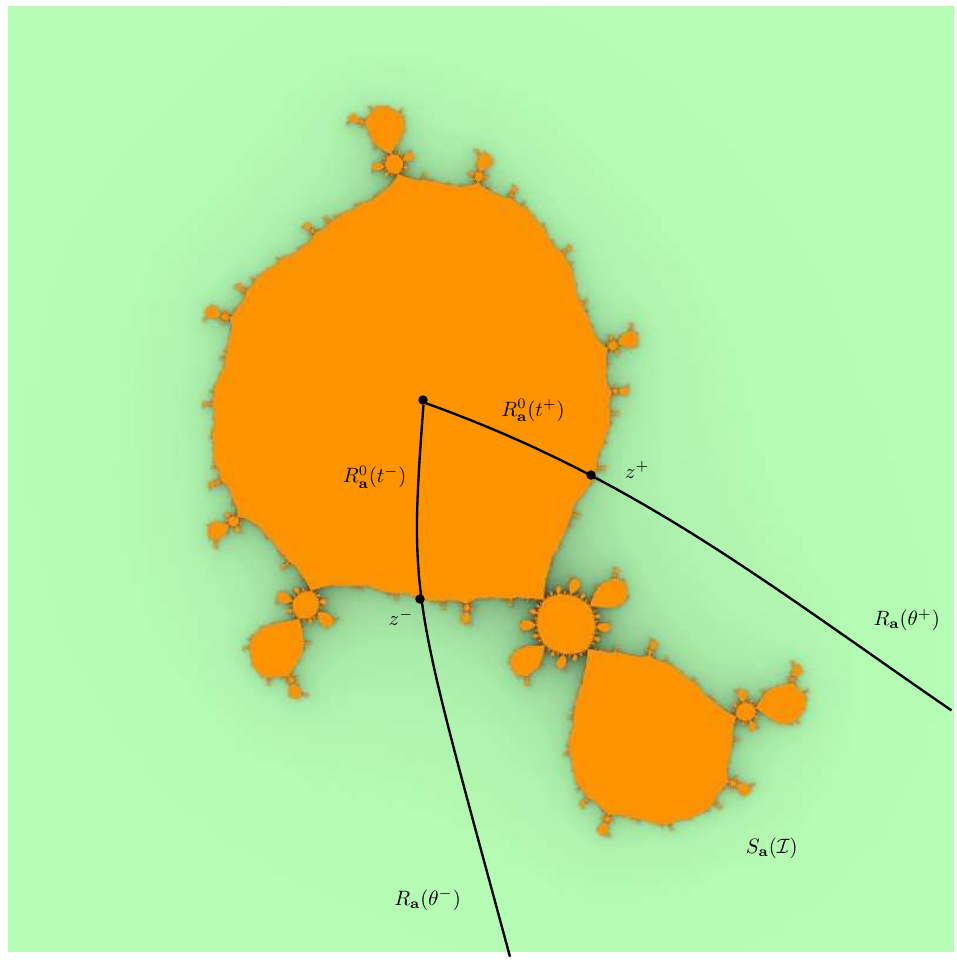}
    \caption{$S_{\a}(\I)$ }
  \end{minipage}
  \hspace{2mm}
  \begin{minipage}{.48\linewidth}
    \includegraphics[width=\linewidth]{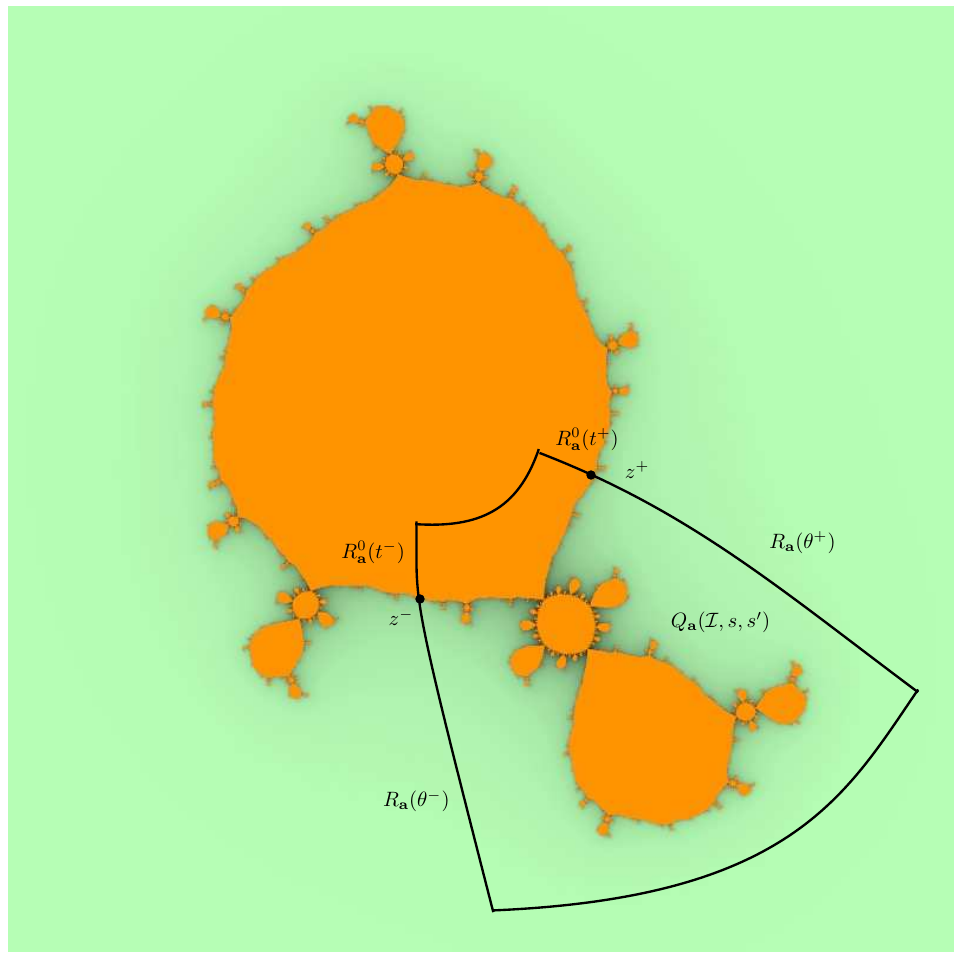}
    \caption{$Q_{\a}(\I,s,s')$}
  \end{minipage}
 \end{center}
\end{figure}



\begin{lem}[quadrilaterals and limbs]\label{land}
    Suppose that $f_\a$ for some $\a \in \AA$ has connected Julia set.
    Then for any $t_0 \in \RZ$ with $s_\a^0(t_0)=1$, the internal ray $\RRR_\a(t_0)$ lands at a point $x \in \partial B_\a$. 
    Moreover, we have
    \begin{equation}\label{limb}
        L_x= \bigcap_{(t^-,t^+)_+\ni t_0} (S_\a(\I) \cap K_\a) \setminus B_\a = \bigcap_{(t^-,t^+)_+\ni t_0, s<1,s'>1}  Q_{\a}(\I,s,s').
    \end{equation}
    where $L_x$ denote the limb with root $x$.
\end{lem}

\begin{proof}

  Let $\gamma: \RZ \to \partial B_\a$ be the parameterization such that $ f_\a^p \circ \gamma(t) = \gamma(t^D) $ holds for $t \in \RZ$.
  Here we have $D\geq d_0$.
  We associate every combinatorics $\I=((t^-,t^+)_+,(\theta^-,\theta^+)_+)$ with an arc $X_\I \subset \RZ$ so that $\gamma(X_\I)=S_\a(\I) \cap \partial B_\a$.

  \begin{clm}\label{clmcontinuity}
     If $\LL(t^-,t^+)_+\to 0$ and $ (t^-,t^+)_+ \ni t_0 $, then $\LL X_\I \to 0$.
    \end{clm}

    \begin{proof}[Proof of Claim \ref{clmcontinuity}]
        
  If $D=d_0$, i.e. $0$ is the unique critical point of $f_\a^p$ in $B_\a$, then $\psi_\a^0$ can be extended to a homeomorphism from $\D$ to $\overline{B_\a}$.
  Thus, $\gamma$ can be defined as $t \mapsto \psi_\a^0(e^{2\pi i t})$.
  In this case we have $X_\I=(t^-,t^+)_+$.
  Hence the conclusion holds.

  Now we consider the case that $D>d_0$.
  For any $n$, since $s_\a^0(t_0)=1$, we can find $\I_n=((t^-_n,t^+_n)_+,(\theta^-_n,\theta^+_n)_+)$ with $ (t^-,t^+)_+ \ni t_0 $ and $\LL(t^-_n,t_n^+)_+$ small enough such that $f_\a^{pn}$ is univalent on $S_\a(\I) \cap  B_\a$.
  It follows that $f_\a^{pn}$ is a homeomorphism on $\gamma(X_{\I_n})$.
  Hence $\tau_D^n$ is a homeomoprhism on $X_{\I_n}$.
  This shows that $\LL X_{\I_n}<D^{-n}$.
  Hence we also deduce $\LL X_\I \to 0$ as $\LL(t^-,t^+)_+\to 0$ for this case.   
  \end{proof}

  Thus, the intersection $\bigcap_{(t^-,t^+)_+\ni t_0 }X_\I$ is a singleton $\{T\}$.  
  Let $x=\gamma(T)$, and 
  $$K:= \bigcap_{(t^-,t^+)_+\ni t_0} (S_\a(\I) \cap K_\a) \setminus B_\a = \bigcap_{(t^-,t^+)_+\ni t_0, s<1,s'>1}  Q_{\a}(\I,s,s'). $$
  Since $L_x \subset Q_\a(\I,s,s')$ for every $\I$, we have $L_x \subset K$.
  For the converse, by Claim \ref{clmcontinuity}, for every $(T^-,T^+)_+ \ni T$, there exists $\I=((t^-,t^+)_+,(\theta^-,\theta^+)_+)$ such that $ (t^-,t^+)_+ \ni t_0 $ and $X_\I \subset (T^-,T^+)_+$.
  Hence $K\subset (S_\a(\I) \cap K_\a) \setminus B_\a \subset L_\a((T^-,T^+)_+,s')$.
  By \eqref{limbdef}, we have $K\subset L_x$.
  Hence $K=L_x$, i.e. \eqref{limb} holds.

  Finally, we show that the internal ray $\RRR_\a(t_0)$ lands at $x$.
  Let $Z$ be the limit set of $\psi_\a^0(se^{2\pi i t_0})$ as $s \to 1$.
  Then it is easy to verify that $Z \subset K\cap \partial B_\a=L_x \cap \partial B_\a =\{ x \}$.
  Therefore, $\RRR_\a(t_0)$ lands at $x$.

\end{proof}

\subsection{Change the dynamical position}\label{subsec34}

Let $\ell>1$ be the least integer such that $f_{\a_0}^\ell(c(\a_0)) \in B_{\a_0}$.
For $\a \in \AA$, we denote $v(\a):=f_\a^\ell(c(\a))$, and $\Phi(\a):=\phi_\a(v(\a))$ if it is defined.
The map $\Phi$ records the dynamical position of the orbit of the critical point $c(\a)$.
We assume that $\Phi(\a_0)=re^{2\pi i t_0}$.

\begin{defi}[critically separable]
    We say a parameter $\a \in \qc(\a_0)$ with $\arg\Phi(\a) = t \in \RZ$ is \emph{critically separable} if for $n\geq 0$ and every critical point $\omega$ of $f_\a^p$, there exists a combinatorics $\I=((t^-,t^+)_+,(\theta^-,\theta^+)_+)$ such that $\tau_{d_0}^n(t) \in  (t^-,t^+)_+$ and $\omega \notin S_\a(\I)$.
Let $\qc^*(\a_0)\subset \qc(\a_0)$ be the class of all $\a \in \qc(\a_0)$ which is critically separable.
\end{defi}

In this part, we apply quasiconformal surgeries to $\a_0$, to change the dynamical position of $v(\a_0)$ to obtain a map $\a_1 \in \qc^*(\a_0)$.

\begin{lem}[rotate the critical point]\label{wring}
    There exists $\a_1 \in \AA \cap \qc^*(\a_0)$ such that $\Phi(\a_1)=re^{2\pi i t_1} $ where $t_1$ is irrational and $\RRR_{\a_1}(t_1)$ land at a singleton limb of $B_{\a_1}$ whose forward orbit omits roots of limbs containing critical points of $f_{\a_1}^p$.
\end{lem}

\begin{proof}
    The main idea of the proof is that we modify the definition of $f_{\a_0}$ in a neighborhood of the critical point $c(\a_0)$ to get a quasi-regular map $F$ such that $F^\ell$ maps $c(\a_0)$ on an internal ray $\RRR_{\a_1}(t_1)$ with some specific angle $t_1$.
    We divide the choice of $t_1$ into two cases.

    {\bf case 1}:
    If $c(\a_0) \notin B_{\a_0}$, then $0$ is the unique critical point in $B_{\a_0}$.
    The B\"ottcher map $\varphi_{\a_0}$ extends to a homeomorphism from $\overline{B_{\a_0}}$ to $\overline{\D}$.
    Pick any irrational $t_1 \in \RZ$ such that the limb $L_x$ with $x=\psi_{\a_0}^0(e^{2\pi i t_1})$ is a singleton and $L_{x_n}$ does not contain a critical point of $f_{\a_0}^p$ for every $n \geq 0$ where $x_n:=f_{\a_0}^{np}(x)$.

    {\bf case 2}:
    If $c(\a_0)\in B_{\a_0}$, then the B\"ottcher map $\phi_{\a_0}$ can be defined at $c(\a_0)$.
    Assume that $\phi_{\a_0}(c(\a_0))=r_0 e^{2\pi i t_0'}$.
    In this case, $\ell$ must be a multiple of $p$.
    We assume that $\ell=\ell'p$.
    Pick an interval $(t_-,t_+)_+$ containing $t_0'$ with $s_{\a_0}^0(t_\pm)=1$ and length $\LL(t_-,t_+)_+=d_0^{-\ell'}$.
    Then we have $\tau_{d_0}^{\ell'}(t_\pm)=t^*$.
    Define a set 
    $$T:=\{ t\in \RZ : \tau_{d_0}^n(t) \notin(t_-,t_+)_+ \mbox{ for all $n$ } \}.$$ 
    Then $T$ is a forward invariant Cantor set under $\tau_{d_0}$ in $\RZ$.
    By Lemma \ref{land}, for every $t \in T$, we have $s_{\a_0}(t)=1$ and $\RRR_{\a_0}(t)$ land at a point in $\partial B_{\a_0}$.
    Let $W$ be the connected component of $B_{\a_0}\setminus (\RRR_{\a_0}(t_-) \cup \RRR_{\a_0}(t_+))$ containing $c(\a_0)$.
    Then for any $t \in T$, the forward orbit of $\RRR_{\a_0}(t)$ omits $W$. 
    Pick irrational $t_1 \in T$ such that the limb $L_x$ with $x=\psi_{\a_0}^0(e^{2\pi i t_1})$ is a singleton and $L_{x_n}$ does not contain a critical point of $f_{\a_0}^p$ for every $n \geq 0$ where $x_n:=f_{\a_0}^{np}(x)$.      

     Thus, in both cases, we find an angle $t_1$.
     Pick $\rho>s_{\a_0}$ and a small simply-connected domain $ V \subset B_{\a_0}$ containing $ v = \psi_{\a_0}^0(re^{2\pi i t_1})$ and $ v(\a_0)$ such that $\phi_{\a_0}(V)\subset \{ w: \rho^{d_0}<|w|<\rho, \arg w \neq t^* \}$.
     Let $U$ be the connected component of $(f_{\a_0}^\ell)^{-1}(V)$  containing $c(\a_0)$.
     Then  $f_{\a_0}^\ell: U \to V$ is a proper map of degree $2$.
     Moreover, for case 2, we have $U\Subset W$.
     Define a quasiregular map $F:\C \to \C$ such that $F|_{\C\setminus U}=f_{\a_0}$ and $F^\ell(c(\a_0))=v$.
     By pulling back the standard complex structure $\sigma_0$ near $\infty$ and $0$ by $F$, we obtain a $F$-invariant complex structure $\sigma$ with bounded dilatation.
     By the Measurable Riemann Mapping Theorem, there exists a quasiconformal map $\xi:\CC \to \CC$ such that $\xi^*(\sigma_0)=\sigma$ and $\xi(\infty)=\infty$, $\xi(0)=0$ and $\xi'(0)=1$.
     It follows that $\xi\circ F\circ \xi^{-1}$ is a polynomial fixing $0$.
     By assigning the critical points, we find $\a_1 \in \AA$ such that $f_{\a_1}= \xi\circ F\circ \xi^{-1}$.

     Notice that $\xi$ forms a conformal conjugation between $F|_{B_{\a_0}(\rho)}=f_{\a_0}$ and $f_{\a_1}$.
     Hence $ \phi_{\a_1}\circ \xi $ forms a conformal conjugation between $f_{\a_0}$ and $z\mapsto z^{d_0}$.
     Combining the fact that $\xi'(0)=1$, we see that $\phi_{\a_1}\circ \xi=\phi_{\a_0}$.
     It follows that
     \begin{align*}
        \Phi(\a_1): &= \phi_{\a_1}(f_{\a_1}^\ell(c(\a_1))) = \phi_{\a_1}\circ f_{\a_1}^\ell \circ \xi(c(\a_0) \\
        &= \phi_{\a_1} \circ \xi \circ F^\ell(c(\a_0)) =\phi_{\a_0}(F^\ell(c(\a_0)))= r e^{2\pi t_1}.
     \end{align*}
     Moreover, since the forward orbit of $\RRR_{\a_0}(t_1)$ omits $W$, $F=f_{\a_0}$ on the orbit of $\RRR_{\a_0}(t_1)$.
     It follows that $s_{\a_1}(t_1)=1$ and $\RRR_{\a_1}(t_1)= \xi(\RRR_{\a_0}(t)) $ lands at $x':=\xi(x)$.
     
     Since $\xi$ is conformal on complement of attracting basin of $0$, $\xi$ forms a conformal conjugation between $f_{\a_0}|_{\Omega_{\a_0}}$ and $f_{\a_1}|_{\Omega_{\a_1}}$, and a quasiconformal conjugation on the Julia sets.
     Hence $\a_1 \in \qc(\a_0)$.
     It follows that $\xi$ preserves the limb structure.
     In particular, $x'=\xi(x) \in \partial B_{\a_1}$ also forms a singleton limb which omits roots of limbs containing critical points of $f_{\a_1}^p$.

     Finally, we show that $\a_1 \in \qc^*(\a_0)$.
     Let $x'_n:=f_{\a_1}^p(x')$, then $\RRR_{\a_1}(\tau_{d_0}^n(t_1))$ lands at $x_n'$.
     Since $x'$ is a singleton limb, we see that every $x_n'$ for $n \geq 0$ forms a singleton limb which omit critical points of $f_{\a_1}^p$.
     For every $n$ and critical point $\omega$ of $f_{\a_1}^p$ outside $B_{\a_1}$, since $\omega \notin L_{x_n'}$, by \eqref{limb} in Lemma \ref{land}, we can find a combinatorics  $\I=((t^-,t^+)_+,(\theta^-,\theta^+)_+)$ such that $\tau_{d_0}^n(t) \in  (t^-,t^+)_+$ and $\omega \notin S_\a(\I) \supset L_{x'_n}$.  
     It remains to separate the critical points of $f_{\a_1}^p$ in $B_{\a_1}$.
     Notice that the only possible critical point of $f_{\a_1}^p$ in $B_{\a_1}\setminus \{ 0 \}$ is $c(\a_1)$ and we are in case 2.
     Since $\tau_{d_0}^n(t_1)) \neq \arg \phi_{\a_1}(c(\a_1))$ for every $n$, we can find a combinatorics $\I =((t^-,t^+)_+,(\theta^-,\theta^+)_+) $ with $ \arg \phi_{\a_1}(c(\a_1)) \notin (t^-,t^+)_+ \ni \tau_{d_0}^n(t_1)$ .
     It follows that $c(\a_1) \notin S_\a(\I) \supset L_{x'_n}$.


     \end{proof}

     Next, we use the stretching deformation introduced by Branner-Hubbard \cite{branner1988iteration} to push the critical point $c(\a_0)$ down along the internal ray with angle $t_1$.
     The stretching deformation allows us to find a parameter ray in $\qc^*(\a_0) \cap \AA $ which has the same dynamical properties as $f_{\a_1}$.

\begin{lem}[stretching operation]\label{stretch}
    There exists a stretching ray 
    $$\mathcal{R}:=\{ \a(s):  0<s<1  \} \subset \qc^*(\a_0) \cap \AA$$
    such that $\a(r)=\a_1$ and $\Phi(\a(s))=se^{2\pi i t_1}$. 
    For every $s \in (0,1)$, there exists a quasiconformal map $\xi_s: \C \to \C$ such that $\xi_s \circ f_{\a_1}=f_{\a(s)} \circ \xi_s$ on $\C$.
    Moreover, $\xi_s$ maps the internal ray $\RRR_{\a_1}(t)$ to $\RRR_{\a(s)}(t)$,  maps the external ray $R_{\a_1}(\theta)$ to $R_{\a(s)}(\theta)$.    
   
\end{lem}

\begin{proof}
     For $s \in (0,1)$, let $\ell_s: \C \to \C$ be the linear map given by $\ell_s(z)=z|z|^{\nu(s)-1}$ where $\nu(s):=\frac{\log s}{\log r}$.
     Then $\ell_s$ commutes with $w \mapsto w^{d_0}$.
     Let $\sigma_s:=\ell_s^*(\sigma_0)$ where $\sigma_0$ denote the standard complex structure.
     Let $\sigma'_{s}$ be the complex structure such that 
    \begin{itemize}
        \item $\sigma'_{s}:=(\phi_{\a_1}^c)^*(\sigma_s)$ on $B_{\a_1}( s_{\a_1})$;
        \item $\sigma_s'=\sigma_0$ outside the attracting basin of $0$;
        \item $f_{\a_1}^*(\sigma'_s)=\sigma_s'$.
    \end{itemize}
     This complex structure $\sigma_s'$ has bounded dilatation and depends continuously on $s \in (0,1)$. 
    By the Measurable Riemann Mapping Theorem, there exists a quasiconformal map $\xi_s:\CC \to \CC$ such that $\xi_s^*(\sigma_0)=\sigma_s'$ and $\xi_s(\infty)=\infty$, $\xi_s(0)=0$ and $\lim_{z\to 0}\frac{\xi_s(z)}{z}=1$.
    It follows that $\xi_s\circ f_{\a_1}\circ \xi_s^{-1}$ is are polynomials fixing $0$ which depends continuously on $s \in \R$.
    Hence we have $\mathcal{R}:=\{ \a(s) : s\in (0,1) \}  $ forms a curve in $\AA$.
    It follows that $\xi_s \circ f_{\a_1}=f_{\a(s)} \circ \xi_s$ holds on $\C$.

    Next, we show that $\phi_{\a(s)}^c\circ \xi_s|_{B_{\a_1} (s_{\a_1} )} =\ell_s \circ \phi_{\a_1}^c$.     
    In fact, consider the map $\phi:= \ell_s \circ \phi_{\a_1}^c \circ (\xi_s|_{B_{\a_1} (s_{\a_1} )})^{-1} $, for $z \in \xi_s^{-1}( B_{\a_1} (s_{\a_1} ) ) $, we have 
    \begin{align*}
        \phi \circ f_{\a(s)}^p (z) &= \ell_s \circ \phi_{\a_1}^c \circ (\xi_s|_{B_{\a_1} (s_{\a_1} )})^{-1}  \circ f_{\a(s)}^p (z)  \\
                                &=  \ell_s \circ \phi_{\a_1}^c \circ f_{\a_1}^p \circ  (\xi_s|_{B_{\a_1} (s_{\a_1} )})^{-1} (z)  \\
                                &= (\ell_s \circ \phi_{\a_1}^c \circ (\xi_s|_{B_{\a_1} (s_{\a_1} )})^{-1}(z))^{d_0} = (\phi(z))^{d_0}.                             
                                \end{align*}
    Moreover, $\phi^*(\sigma_0)= (\ell_s \circ \phi_{\a_1}^c \circ (\xi_s|_{B_{\a_1} (s_{\a_1} )})^{-1})^*(\sigma_0) = (\xi_s|_{B_{\a_1} (s_{\a_1} )})^{-1})^* (\phi_{\a_1}^c)^* \ell_s^*(\sigma_0) =\sigma_0 $.
    It follows that $\phi$ is a B\"ottcher coordinate $\phi_{\a(s)}$ for $f_{\a(s)}$ in $\xi_s^{-1}( B_{\a_1} (s_{\a_1} ) )$.
    Since $\lim_{z\to 0}\frac{\xi_s(z)}{z}=1$ and $\xi_s$ is a conjugation, $\xi_s $ maps the internal ray $\RRR_{\a_1}(t)$ with $s_{\a_1}^0(t)=1$ to $\RRR_{\a(s)}(t_1)$.
    In particular, $s_{\a(s)}^0(t)=1$.
    Moreover, we have 
     \begin{align*}
        \Phi(\a(s)) &=\phi_{\a(s)}^c \circ f_{\a(s)}^\ell (c(\a(s)))  = \phi_{\a(s)}^c \circ f_{\a(s)}^\ell \circ \xi_s(-c(\a_1)) \\
        &= \phi_{\a(s)}^c \circ \xi_s  \circ f_{\a_1}^\ell(c(\a_1)) =\ell_s \circ \phi_{\a_1}^c \circ f_{\a_1}^\ell(c(\a_1)) \\
        &=\ell_s \circ \Phi (\a_1)=r^{\nu(s)} e^{2\pi i t_1} = s e^{2\pi i t_1} .
    \end{align*}

    Since $\sigma_s'=\sigma_0$ on $\Omega_{\a_1}$, then $\xi_s|_{\Omega_{\a_1}}$ is a conformal conjugacy between $f_{\a_1}$ and $f_{\a(s)}$.
    Thus, $\varphi_{\a(s)}\circ \xi_s$ conjugates $f_{\a_1}$ and $z\mapsto z^d$.
    We claim that $ \varphi_{\a(s)}\circ \xi_s $ is the B\"ottcher map $\varphi_{\a_1}$ on $\Omega_{\a_1}$.
    By the uniqness of the B\"ottcher map, it suffices to show that $ \varphi_{\a(s)}\circ \xi_s $ maps some external ray $R_{\a_1}(\theta)$ to $R_{\a(s)}(\theta)$.
    Pick a $t \in \QZ$ such that the internal ray $\RRR_{\a_1}(t)$ lands at a repelling periodic point $z_0$ of some period $q$, which is the landing point of the external ray $R_{\a_1}(\theta)$.
    By Lemma \ref{stablerepelling}, for $s$ close to $r$, $\RRR_{\a(s)}(t)$ and $R_{\a(s)}(\theta)$ both land at common periodic point $z_0(s)$ of period $q$.
    By the continuity of $\xi_s$, we have $\xi_s(z_0)=z_0(s)$ and hence $\xi_s(R_{\a_1}(\theta))=R_{\a(s)}(\theta)$.
    It follows that $ \varphi_{\a(s)}\circ \xi_s =\varphi_{\a_1}$ on $\Omega_{\a_1}$ for $s$ close to $r$.
    By a typical connectedness argument, we conclude that $ \varphi_{\a(s)}\circ \xi_s =\varphi_{\a_1}$ on $\Omega_{\a_1}$ holds for $s \in (0,1)$.
    Hence $\xi_s$ maps the external ray $R_{\a_1}(\theta)$ to $R_{\a(s)}(\theta)$.
    Since $F=f_{\a_1}$ on the Julia set, we have $\mathcal{R} \subset \qc(\a_1) \cap \AA$.
    Combining the fact that $\xi_s$ preserves internal rays and external rays, we see that  $\mathcal{R} \subset \qc^*(\a_1) \cap \AA = \qc^*(\a_0) \cap \AA $.
   
     \end{proof}

\section{Dynamic and Lamination of the Limit Maps}\label{sec4}

In this section, we finish the proof of the main result (Theorem \ref{main}).
In fact, we will show that every limit point of $\a(s)$ as $s \to 1$ belongs to $\comb(\a_0)\setminus \qc(\a_0)$.

\subsection{The limit map} 
Let $\mathcal{X}$ denote the limit set of $\a(s)$ as $s \to 1$.
Then $\mathcal{X}\neq \emptyset$.
Pick a limit point $\b \in \mathcal{X}$, we may suppose that $\a_n:=\a(s_n) \to \b$ and $s_n \to 1$ as $n\to \infty$.
The main purpose of this part is to characterize the dynamics of $f_\b$.
Since $c(\a_n) \to c(\b)$ as $n \to \infty$, then we have $f_{\b}^{m_k}(c_k(\b)) = f_{\b}^{n_k}(c_k(\b))$ for $k \in I_F$ and $f_{\b}^p(0)=0$.
Hence $\b \in \AA$, i.e., the critical relations of critical points $c_k$ for $k \in I_F$ of the limit map is preserved.

\begin{lem}[critical points in attracting cycles]\label{attracting}
    Critical points of $f_{\b}$ in attracting cycles are contained in $\{ c_k(\b) : k \in I_F \} \cup \{ 0 \}$.
\end{lem}

\begin{proof}
    It suffices to show that every $\omega(\b) \in \{ c(\b) \} \cup \{ c_k(\b): k \in I_\infty\}$ is not in attracting cycles.
    If  $\omega(\b)$ belongs to the attracting basin of some periodic point $x$ with some period $q$.
    There exists a neighborhood $U$ of $x$ and a neighborhood $\mathcal{U}$ of $\b$ such that $f_{\a}^q(U)\Subset U$ for $\a \in \mathcal{U}$.
    Thus $U$ is contained in an attracting basin of $f_\a$.
    Since $\omega(\b)$ is attracted by $x$, there exists $m$ such that $f_{\b}^m(\omega(\b))\in U$.
    By shrinking $\mathcal{U}$, we may assume that $f_\a^m(\omega(\a)) \in U$.
    Hence $\omega(\b)$ cannot be $c_k(\b)$ for any $k \in I_\infty$ since
    $c_k(\a)$ is in the Julia set for $\a$ on the stretching ray.  

    It remains to exclude the case that $\omega(\b)=c(\b)$.
    By above discussion, we see that $x$ must be $0$.
    Hence we may assume that $f_{\b}^m(c(\b)) \in B_\b$ with $|\phi_\b|( f_{\b}^m(c(\b)) )<\rho$ for some $0<\rho<1$.
    Then there exists a neighborhood $\mathcal{U}$ of $\b$ such that $|\phi_\a|( f_{\a}^m(c(\a)) )<\rho$ for every $\a \in \mathcal{U}$.
    On the other hand, by Lemma \ref{stretch}, we have $|\phi_{\a_n}|( f_{\a_n}^m(c(\a_n)) ) \to 1$ as $n \to \infty$.
    This is a contradiction.
    
\end{proof}
 
For $\a \in \AA$ and $n \geq 0$, denote $\crit_n(\a)$ denote the set of critical points of $f_\a^n$.
It admits the following decomposition:
$$\crit_n(\a)=\crit_{n,0}(\a) \cup \crit_{n,c}(\a) \cup \crit_{n,2}(\a) \cup \cdots \cup \crit_{n,d-2}(\a),$$
where $\crit_{n,k}(\a):=\{ z : f_\a^j(z)=c_k(\a) \mbox{ for some $0\leq j < n$} \}$ for $2\leq k \leq d-2$, and  $\crit_{n,c}(\a):=\{ z : f_\a^j(z)=c(\a) \mbox{ for some $0\leq j < n$} \}$.

An element in $\crit_n(\a)$ may not be well-defined in a neighborhood of $\b$ in $\AA$.
However, it is a well-defined continuous map on $\mathcal{R} \cup \{ \b\}$.
Thus, we can still use the notation $\omega(\a)$ to represent an element in $\crit_n(\a)$ for $\a \in \mathcal{R} \cup \{ \b\}$.

\begin{lem}[free critical orbit]\label{free}
        The internal ray $\RRR_{\b}(t_1)$ lands at $v(\b)$.
\end{lem}

\begin{proof}

    Suppose that the internal ray $\RRR_{\b}(t_1)$ lands at $x \in \partial B_{\b}$.
    First we show that $v(\b)$ is contained in the limb $L_x$ with respect to $B_{\b}$.
    We prove it by contradiction.
    By Lemma \ref{land}, if $v(\b)$ is not in this limb, then we can find a combinatorics $\mathcal{I}=( (t^-,t^+)_+, (\theta^-,\theta^+)_+ )$ such that  $(t^-,t^+)_+ \ni t_1$ such that $v(\b)\notin S_{\b}(\I)$.
    By the holomorphic motion of $\partial S_\a(\I)$ (Fact \ref{fact}), we see that for $\a \in \mathcal{U}\cap \mathcal{R}$, we have $v(\b) \notin S_\a(\I)$.
    Notice that $S_\a(\I)$ still contains the internal ray $\RRR_\a(t_1)$.
    Hence we find a contradiction since we have $v(\a) \in \RRR_\a(t_1)$ for $\a \in \mathcal{R}$.

    To prove the lemma, it suffices to show that the limb $L_x$ is a singleton.
    By Theorem \ref{roeschyin}, we only need to show that $L_{x_n}$ does not contain any critical point of $f_{\b}^p$ for every $n$ where $x_n:=f_{\b}^{np}(x) \in \partial B_{\b}$.
    We prove it by contradiction, if $L_{x_n}$ contains a critical point $\omega(\b) \in \crit_p(\b)$, then for every combinatorics $\I=((t^-,t^+)_+,(\theta^-,\theta^+)_+)$ with $ (t^-,t^+)_+ \ni t_1 $, we have $\omega(\b) \in S_\b(\I)$.
    By Fact \ref{fact}, $\partial S_\a(\I)$ moves holomorphically for $\a \in \mathcal{U}$.
    It follows that for $\a \in \mathcal{U}\cap \mathcal{R}$, there exists $\omega(\a) \in S_\a(\I) \cap \crit_p(\a)$.
    This contradicts the fact that $\a \in \mathcal{R} \subset \qc^*(\a_0)$ is critically separable.

\end{proof}

\subsection{Puzzles and polynomial-like maps}\label{subsec42}

To characterize the dynamics of $c_k(\b)$ for $k \in I_\infty$, we need to use the puzzle technique to construct generalized polynomial-like maps near the critical point $c_k(\b)$.

Pick $t^* \in \Q/\Z$ such that the internal ray $\RRR_{\b}(t^*)$ lands at a repelling periodic point $z^*$ on $\partial B_{\b}$.
Pick $\theta^* \in \QZ$ such that the external ray $R_{\b}(\theta^*)$ lands at $z^*$
Define a forward invariant graph of $f_\b$:
$$\varGamma(t^*):=\bigcup_{n\geq 0} f_{\b}^n(\RRR_{\b}(t^*)\cup R_{\b}(\theta^*)\cup \{ z^*\}),$$ 
and a forward invariant domain
$$X:= \C\setminus \left(|\varphi_{\b}|^{-1}([s',+\infty)) \cup \bigcup_{0\leq k <p} f_{\b}^k(|\phi_{\b}|^{-1}([0,s]))\right).$$
For every $n \geq 0$, connected components of $f_{\b}^{-n}(X\setminus \varGamma(t^*))$ are called \emph{puzzle pieces} of depth $n$.
For every $n$, define $\varGamma_n:=f_{\b}^{-n}(\varGamma(t^*))$, and $G_n:=f_{\b}^{-n}(\varGamma(t^*)\cup \partial X) \supset \varGamma_n$.
It follows that for every puzzle piece $P_n$ of depth $n$, we have $ G_n \subset G_{n+1}$, and $\partial P_n \subset G_n$,.

For any $z \in K_{\b}$ outside the basin of $0$, if the forward orbit of $z$ under $f_{\b}$ omits $\varGamma(t^*)$, then for every $n$, there exists a unique puzzle piece $P_n(z)$ of depth $n$ containing $z$. 
Define the \emph{impression} of $z$ by 
$K(z):=\bigcap_n \overline{P_n(z)}$.
If $z$ belongs to some limb $L_x$, then one may verify that $K(z)\subset L_x$.


\begin{fact}[holomorphic motion of puzzles]\label{fact2}

    For every $n$, there exists a neighborhood $\mathcal{P}_n \subset \AA$ of $\b$ and a holomorphic motion $h_n : \mathcal{P}_n \times G_n \to \C$, $(\a,z)\mapsto h_n^\a(z)$ such that for every $0\leq m < n$, the following diagram commutes.
\begin{equation*}
    \xymatrix{\ar @{} [dr]
   G_{n} \ar[d]^{f_{\b}^{n-m}} \ar[r]^{h^{\a}_n} & G_n^\a \ar[d]^{f_{\a}^{n-m}} \\
    G_m \ar[r]^{h_n^\a}        &  G_m^\a      }
\end{equation*}
where $G_n^\a:=h_n^\a(G_n)$.
For every puzzle piece $P_n$ of depth $n$, denote $P_n^\a$ be the bounded complementary component of $h_n^\a(\partial P_n)$.

Suppose that $f^n_{\b}(\omega(\b)) \in P_n$ for some $\omega(\b) \in \crit_{m,k}(\b)$ with $m\geq 0$, $n\geq 0$ and $k \geq 2$, then we have $f^n_\a(\omega(\a)) \in P^\a_n$ for every $\a \in \mathcal{P}_n\cap \mathcal{R}$.
\end{fact}

\begin{proof}
    The first assertion follows directly from Lemma \ref{stablerepelling} and the lifting property.
    We only prove the second assertion.
    Suppose that $f^n_{\b}(\omega(\b)) \in P_n$ for some $\omega(\b) \in \crit_{m,k}(\b)$ with $m\geq 0$, $n\geq 0$ and $k \geq 2$.
    By the holomorphic motion of $\partial P_n$ and the fact that $ f^n_\a(\omega(\a)) \to f_{\b}^n(\omega(\b))$ as $\a \to \b$ on $\mathcal{R}$, there exists a neighborhood $\mathcal{U}\subset \mathcal{P}_n$ such that $f^n_\a(\omega(\a)) \in P^\a_n$ for every $\a \in \mathcal{U}\cap \mathcal{R}$.
    Fix $\a' \in \mathcal{U} \cap \mathcal{R}$, for any $\a \in \mathcal{R}\cap \mathcal{P}_n$, 
    by Lemma \ref{stretch}, there exists a quasiconformal conjugation $\xi$ such that $\xi \circ f_{\a'}=f_\a \circ \xi$.
    Denote $\varGamma_n^\a:=h_n^\a(\varGamma_n)$.
    Moreover, $\xi$ preserves external rays and internal rays.
    It follows that $\xi(\varGamma_n^{\a'})=\varGamma_n^\a$.
    Hence $f_\a^n(\omega(\a))$ belongs to the connected component of $\C \setminus\varGamma_n^\a$ containing $P_n^\a$.
    Notice that for $\a \in \mathcal{R}\cap \mathcal{P}_n$, $\omega(\a)$ either belongs to the Julia set or is the center of some attracting component.
    We conclude that $f^n_\a(\omega(\a)) \in P^\a_n$  for every $\a \in \mathcal{P}_n\cap \mathcal{R}$.

\end{proof}




\begin{lem}[no indifferent cycle]\label{noindifferent}
    $f_{\b}$ do not possesses any indifferent cycle.
\end{lem}

\begin{proof}
    We show that $f_{\b}$ has no indifferent cycle by contradiction.
    Suppose that $f_{\b}$ admits an indifferent periodic point $\beta$ of period $q_0$.
    Pick a multiple $q$ of $q_0$ such that $P_q(\beta)\Subset P_0(\beta)$.
    It follows that for every $n$, we have $P_{n+q}(\beta)\Subset P_n(\beta)$, and $f_{\b}^q: P_{n+q}(\beta) \to P_n(\beta)$ is a polynomial-like map of some degree $D\geq 1$.
    If $D=1$, then $\beta$ must be a repelling periodic point which is a contradiction.

    By Lemma \ref{free}, we have $K(c(\b))=L_x=\{  c(\b)\}$ is a singleton impression.
    It follows that for $m$ large, we have $P_m(\beta)\cap \crit_{q,c}(\b)=\emptyset$.
    Pick $m$ large such that $P_{m}(\beta) \cap \crit_q(\b) = K(\beta) \cap  \crit_q(\b)\subset \bigcup_{k\geq 2} \crit_{q,k}(\b)$.
    Let $U:=P_{m+q}(\beta)$, $V:=P_m(\beta)$.
    Then the polynomial-like map $\f:=(f_{\b}^q,U,V)$ is a polynomial-like map of degree $D \geq 2$ with the filled-in Julia set $K(\f)=K(\beta)$.

    By Fact \ref{fact2}, there exists a neighborhood $\mathcal{U}:=\mathcal{P}_{m+q}$ such that $U_\a$ and $V_\a$ are well-defined, and $\f_\a:=(f_\a^q, U_\a ,V_\a )$ also forms a polynomial-like map of degree $D$ for $\a \in \mathcal{U}$.
    Recall that $\{ \a_n\}\subset \mathcal{R}$ is a sequence convergent to $\b$.
    We may assume that $\a_n \in \mathcal{U}$ for $n\geq m$.
    We claim that every $\f_{\a_n}$ with $n \geq m$ has connected filled-in Julia set $K_n$.
    For every critical point $\omega(\a_n) \in \mathrm{Crit}_q(\a_n) \cap U_{\a_n} $,
    by Fact \ref{fact}, we see that $f_{\a_n}^{iq}(\omega(\a_n)) \in U_{\a_n}$ for every $i$ and $n \geq m$.
    Thus, the filled-in Julia set $K_n$ is also connected for $n \geq m$.

     By the Douady-Hubbard Straightening Theorem, for every $n \geq m$, there exist a quasiconformal map $\eta_n: U_{\a_n} \to \C$ and a polynomial $g_n$ such that $f_{\a_n}^q$ is hybrid conjugated to a polynomial $g_n$ by $\eta_n$ up to an affine conjugation.
     By Lemma \ref{stretch}, there exists a quasiconformal conjugation $\zeta_n: \C \to \C$ between $f_{\a_{m}}$ and $f_{\a_n}$.

     \begin{clm}\label{clmhybrid}
         $\zeta_n : K_{m} \to K_n$ forms a hybrid equivalence map from $\f_{\a_m}$ to $\f_{\a_n}$.
     \end{clm}

     \begin{proof}[Proof of Claim \ref{clmhybrid}]
     It suffices to show that $\zeta_n(K_{m})=K_n$.
     For every $z \in K_{m}$, we have $f_{\a_{m}}^{q\ell}(z) \in U_{m}$ for any $m$.
     Denote $z':=\zeta_n(z)$, then by the conjugation property, we have $f_{\a_n}^{q\ell}(z')=\zeta_n\circ f_{\a_{m}}^{q\ell}(z) \in \zeta_n(U_{m}) $ for every $\ell$. 
     Notice that $\zeta_n$ preserves all external rays and internal rays in the attracting basin of $0$, then the symmetric difference between $\zeta_n(U_{m})$ and $U_n$ are contained in the attracting basin of $0$.
     It follows that $f_{\a_n}^{mq}(z') \in U_n $ for every $m$.      
     Hence $z' \in K_n$.
     This shows that $\zeta_n(K_{m})\subset K_n$.
    The proof of the converse is the same and hence omitted.
     \end{proof}

     Thus, the composition $\eta_{m}\circ\zeta_n^{-1}$ forms a hybrid equivalence map between $f_{\a_n}^q$ and $g_{m}$.
    By the uniqueness of the renormalized maps, we may assume that $g_{n}=g_{m}:=g$.
    Moreover, since $\partial V_\a \cup \partial U_\a$ moves holomorphically for $\a \in \mathcal{U}$, $\mathrm{mod}(V_\a\setminus \overline{U_\a})$ is uniformly bounded below for $\a \in \mathcal{U}$.
    It follows that the dilatation of $\eta_n$ have a uniform bound for $n \geq m$.
    Hence, we may assume that $\eta_n$ converges to a quasiconformal map $\eta$.
    It forms a hybrid conjugation between $f_{\b}|_{K(\beta)}$ and $g$.
    It follows that $f_{\b}^{q}|_{K(\beta)}$ and $f_{\a_{m}}^q|_{K_{m}}$ are hybrid conjugated.
    Since $f_{\b}$ has an indifferent cycle, we conclude that $f_{\a_{m}}$ also has an indifferent cycle.
    This is a contradiction. 

    \end{proof}

   Now we can conclude the dynamics of the limit map $f_\b$ in the following.

\begin{prop}[dynamics of $f_{\b}$]\label{limit}
    $f_{\b}$ has no indifferent cycle.
    The critical points $c(\b)$ and $c_k(\b) : k \in I_\infty$ are in the Julia set with infinite forward orbit.
\end{prop}

\begin{proof}

    By Lemma \ref{attracting}, Lemma \ref{free} and Lemma \ref{noindifferent}, it remains to show that every $c_k(\b)$ for $k \in I_\infty$ has infinite forward orbit.
    We prove it by contradiction, suppose that $z(\b):=f_{\b}^m(c_k(\b))$ is a periodic point of period $q$.
    Pick an $n$, denote $V_1:=P_n(z(\b))$, $U_1:=P_{n+q}(z(\b))$, and $U_2:=P_{n+q+m}(c_k(\b))$, $V_2:=P_{n+m}(c_k(\b))$.
    Denote $\f_1:=(f_\b^q,U_1,V_1)$ and $\f_2:=(f_\b^m,U_2,V_2)$.
    Then $(\f_1,\f_2)$ forms a generalized polynomial-like map of some degree $(d_1,d_2)$ with $d_1\geq 1$, $d_2\geq 2$ with filled-in Julia set $(K_1,K_2)$.
    Similar to the proof of Lemma \ref{noindifferent}, we can show that 
    \begin{itemize}
        \item $K_1 \cap \crit_q(\b) \subset \bigcup_{k\geq 2}\crit_{q,k}(\b)$ and $K_2 \cap \crit_m(\b) \subset \bigcup_{k\geq 2}\crit_{m,k}(\b) $.
        \item $(\f_1,\f_2)$ has connected filled-in Julia set $(K_1,K_2)$.
    \end{itemize}
    By Fact \ref{fact2}, there exists a neighborhood $\mathcal{U}:=\mathcal{P}_{n+q+m}$ such that $U_1^\a,V_1^\a,U_2^\a,V_2^\a$ are defined for $\a \in \mathcal{U}$ and  $(\f_{\a,1},\f_{\a,2})$ also forms a generalized polynomial-like map of some degree $(d_1,d_2)$ where $\f_{\a,1}:=(f_\a^q,U_1^\a,V_1^\a)$ and $\f_{\a,2}:=(f_\a^m,U_2^\a,V_2^\a)$.
    Similar to the proof of Lemma \ref{noindifferent}, by Fact \ref{fact2}, we can show that $(\f_{\a,1},\f_{\a,2})$ also has connected filled-in Julia set.
    Assume that $\a_n \in \mathcal{U}$ for $n\geq N$.
    Then for $n\geq N$, the filled-in Julia set of $(\f_{\a_n,1},\f_{\a_n,2})$ is denoted by $(K_1^{n},K_2^{n})$.
    Thus, by Proposition \ref{straightening}, 
    $(\f_{\a_n,1},\f_{\a_n,2})$ is hybrid conjugated to $(g_{n,1},g_{n,2})$ by some hybrid maps $(\eta_n,\iota_n)$. 
    On the other hand, by Lemma \ref{stretch}, since $f_{\a_N}$ and $f_{\a_n}$ is quasiformally conjugated, we may assume that $(g_{n,1},g_{n,2})=(g_{N,1},g_{N,2})=(g_1,g_2)$ is a constant.
    Thus, we have 
    \begin{equation*}
        \eta_n \circ f_{\a_n}^q = g_1 \circ \eta_n, \quad \eta_n \circ f_{\a_n}^m = g_2 \circ \iota_n .
    \end{equation*}
    Therefore, on $K_2^n$, we have
    \begin{equation}\label{eqorbit}
        \eta_n \circ f_{\a_n}^{jq+m}= g_1^j\circ g_2 \circ \iota_n  \forall j \geq 0
    \end{equation}
    Since $ \{ f_{\a_n}^{jq+m}(c_k(\a_n)) : j \geq 0 \} $ is an infinite set, then $\{ g_1^j \circ g_2(\omega_n) : j\geq 0 \}$ is also an infinite set where $\omega_n:=\eta_n(c_k(\a_n))$ is a critical point of $g_2$.
    Hence we may assume that $\omega_n=\omega$ is a constant.
    Since $\mathrm{mod}(V_1^{\a_n} \setminus \overline{U_1^{\a_n}})$ is uniformly bounded below, we may assume that $(\eta_n,\iota_n)$ converges to $(\eta,\iota)$ where $\eta,\iota$ are quasiconformal maps.
    By taking limit in \eqref{eqorbit}, it follows that 
    $$
         f_{\b}^{jq+m}(c_k(\b))= \eta^{-1} \circ g_1^j \circ g_2 \circ \iota(\omega),\quad \forall j \geq 0
    $$
    forms an infinite set.
    It contradicts the assumption that $c_k(\b)$ has finite orbit.


\end{proof}

\subsection{Rational lamination}

In this part, we prove our main result Theorem \ref{main}.

\begin{cor}[local holomorphic motion]\label{localholo}
    For any $\theta \in \Q/\Z$, there exists a neighborhood $\mathcal{U}_\theta \subset \AA$ of $\b$ such that $s_\a(\theta)=1$ for every $\a \in \mathcal{U}_\theta$.
    Thus, there exists a holomorphic motion $h: \mathcal{U}_\theta \times \overline{R_{\b}(\theta)} \to \C$ such that $h(\a, \overline{R_{\b}(\theta)}) = \overline{R_\a(\theta)}$.
\end{cor}

\begin{proof}
    For every $\theta \in \Q/\Z$, the external ray $R_\b(\theta)$ lands at a parabolic or repelling preperiodic point $z_0$.
    By Lemma \ref{noindifferent}, $z_0$ must be repelling. 
    There exists a minimal $m$ such that  $\theta_m:=\tau_d^m(\theta)$ is periodic.
    By Lemma \ref{stablerepelling}, there exists a neighborhood $\mathcal{U}$ of $\b$ such that $s_\a(\theta_m)=1$ for $\a \in \mathcal{U}$.
    Moreover, $R_\a(\theta_m)$ moves holomorphically for $\a \in \mathcal{U}$.
    We will show that there exists a neighborhood $\mathcal{U}_\theta$ of $\b$ such that $s_\a(\theta)=1$ for $\a \in \mathcal{U}_\theta$.    
    We prove it by contradiction.    
    Suppose that there exists $\a_n' \to \b$ in $\AA$ such that $s_{\a'_n}(\theta)>1$.
    Then for every $n$, $w_n:=\psi_{\a'_n}(s_{\a'_n}(\theta)e^{2\pi i \theta})$ is a precritical point.
    Since $s_{\a'_n}(\theta_m)=1$, we see that $f_{\a'_n}^m(w_n) \in R_{\a'_n}(\theta_m)$ is a post critical point.
    Thus, $f_{\a'_n}^m(w_n)$ must be $f_{\a'_n}^\ell(\omega(\a'_n))$ for some $1\leq \ell \leq m$ and $\omega(\a'_n) \in \crit(\a'_n)$.
    By taking a subsequence of $\a_n'$, we may assume that $\ell$ and $\omega$ are independent of $n$.
    By taking the limit, we see that $f_{\b}^\ell(\omega(\b)) \in \overline{R_\b(\theta_m)}$.
    Thus $ f_{\b}^\ell(\omega(\b)) $ must be the landing point of $R_\b(\theta_m)$.
    By Proposition \ref{limit}, we must have $\omega(\b)=c_k(\b)$ for some $k \in I_F$.
    But, for $\a \in \AA$, $c_k(\a)$ is in the Julia set or a center of some attracting component.
    It contradicts that $f_{\a_n'}^m(c_k(\a_n')) \in R_{\a_n'}(\theta_m)$.

\end{proof}

Using Corollary \ref{localholo}, we can finish the proof of our main result Theorem \ref{main}.

\begin{proof}[Proof of Theorem \ref{main}]
    We show that $\L_\Q(f)=\L_\Q(f_{\b})$.
    For any $(\theta,\theta') \in \L_\Q(f_{\b})$, $R_{\b}(\theta)$ and $R_{\b}(\theta')$ land at a common point.
    By Corollary \ref{localholo}, for $\a \in \mathcal{U}_{\theta} \cap \mathcal{U}_{\theta'}$, two external rays $R_\a(\theta)$ and $R_\a(\theta')$ also land at a common point.
    Since $\a_n \in \mathcal{U}_{\theta} \cap \mathcal{U}_{\theta'} $ for  large $n$, we have $(\theta,\theta') \in \L_\Q(f_{\a_n})$.
    Hence $ (\theta,\theta') \in \L_\Q(f) $.

    For any $(\theta,\theta') \in \L_\Q(f)$, we show that $(\theta,\theta') \in \L_\Q(f_{\b})$ by contradiction.   
    Suppose that $(\theta,\theta') \notin \L_\Q(f_{\b})$, then by Corollary \ref{localholo},  for $\a \in \mathcal{U}_{\theta} \cap \mathcal{U}_{\theta'}$, two external rays $R_\a(\theta)$ and $R_\a(\theta')$ do not land at a common point. 
    Hence for $n$ large, we have $(\theta,\theta') \notin \L_\Q(f_{\a_n})$.
    Then we get $ (\theta,\theta') \notin \L_\Q(f) $ which is a contradiction.

    By above discussion, we conclude that $\L_\Q(f)=\L_\Q(f_{\b})$.
    But $f$ and $f_{\b}$ cannot be conjugated since they have different number of Julia critical points.
    Hence $f$ is not combinatorially rigid.

    \end{proof}

\section{Combinatorial Rigidity for Hyperbolic Maps}\label{sec5}

In this section, we prove Theorem \ref{hyperbolic}.
Let $f$ be a hyperbolic map of disjoint type.
We will use the following two facts.

\begin{fact}[center]\label{pcf}
    There exists a hyperbolic post-critically finite map $f_0 \in \qc(f)$.
\end{fact}

The proof of Fact \ref{pcf} is implied in the proof of Lemma \ref{centralize}.
In fact, the map $f_0$ is the center of the hyperbolic component containing $f$.

The following well-known result concerning the rigidity of hyperbolic post-critically finite polynomials.

\begin{fact}[rigidity]\label{rigidity}
    Suppose that $f,g \in \P$ are hyperbolic post-critically finite polynomials such that $\L_\Q(f)=\L(g)$, then $f=g$.
\end{fact}

Now we present the proof of Theorem \ref{hyperbolic}.

\begin{proof}[Proof of Theorem \ref{hyperbolic}]
    If $f$ is not of disjoint type, then there exists a attracting cycle which attract at least two critical points (counted with multiplicity).
    By Theorem \ref{main}, $f$ is not combinatorially rigid.

    In the following, we show that every hyperbolic polynomial $f \in \P$ of disjoint type is combinatorially rigid.
    Let $f$ be a hyperbolic map of disjoint type, we need to show that $\comb(f)\subset \qc(f)$.
    
    By Fact \ref{pcf}, we may assume that $f$ is post-critically finite.
    Then $f$ has $d-1$ periodic critical point $c_1,c_2,\dots,c_{d-1}$ with period $p_1,p_2,\dots,p_{d-1}$.
    For $1\leq k \leq d-1$, let $B_k$ denote the attracting component containing $c_k$.
    Let $\partial B_k^* \subset \partial B_k$ be the set of $x \in \partial B_k$ which is the root of a non-trivial limb $L_{k,x}$.
    By Theorem \ref{roeschyin}, for every $x \in \partial B_k^*$, there exists $(\theta_{k,x},\theta'_{k,x}) \in \L_\R(f)$ with $\theta_{k,x}\neq \theta_{k,x}'$ such that $R_f(\theta_{k,x})$ and $R_f(\theta'_{k,x})$ land at $x$ and separate the limb $L_{k,x}$ away from $B_k$.
    Let $\Omega_{k,x}$ be the connected of $\C \setminus (\overline{R_f(\theta_{k,x})} \cup \overline{R_f(\theta'_{k,x})})$ containing $B_k$.
    Then we have $L_{k,x} \cap \Omega_{k,x} =\emptyset$ and $\bigcap_{x \in \partial B_k^*}\overline{\Omega_{k,x}} \cap K(f)=\overline{B_k}$.
    Let $E_{k,x}\subset \RZ$ be the arc such that for $\theta \in E_{k,x}$, the external ray $R_f(\theta)$ belongs to $\Omega_{k,x}$.
    Let 
    $E_k := \bigcap_{x \in \partial B_k} \overline{E_{k,x}} .$

    \begin{clm}\label{rational}
       For every $1\leq k \leq d-1$ and $x \in \partial B_k^*$, we have  $(\theta_{k,x},\theta_{k,x}') \in \L_\Q(f)$.
    \end{clm}

    \begin{proof}
    
        Since $(\theta_{k,x},\theta'_{k,x}) \in \L_\R(f)$, it suffices to show that $ \theta_{k,x},\theta'_{k,x} \in \QZ $.
        We prove it by contradiction.
        Suppose that there exist $k$ and $x \in \partial B_k^*$ such that $\theta_{k,x},\theta_{k,x}' \in \RZ \setminus (\QZ)$.
        It follows that $x$ must have infinite orbit under $f^{p_k}$.
        Let $x_j:=f^{jp_k}(x)$.
        By Theorem \ref{roeschyin}, since limbs are disjoint, for $j$ large, the limb $L_{k,x_j}$ contains no critical point of $f^{p_k}$.
        Let $E'_{k,x_j}:=\RZ \setminus \overline{E_{k,x_j}}$.
        Then we have $\LL(E'_{k,x_{j+m}})=d ^m\LL(E'_{k,x_{j}})$ for $j$ large.
        This is clearly impossible.

    \end{proof}

     \begin{figure}[h]\label{fig1}
  \begin{center}
   \vspace{2mm}
   \begin{minipage}{.48\linewidth}
    \includegraphics[width=\linewidth]{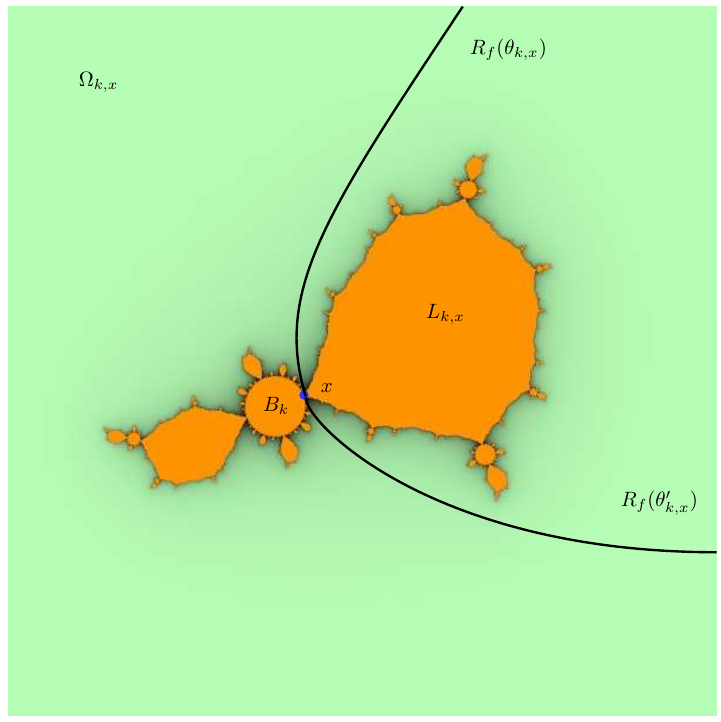}
    \caption{$K(f)$ }
  \end{minipage}
  \hspace{2mm}
  \begin{minipage}{.48\linewidth}
    \includegraphics[width=\linewidth]{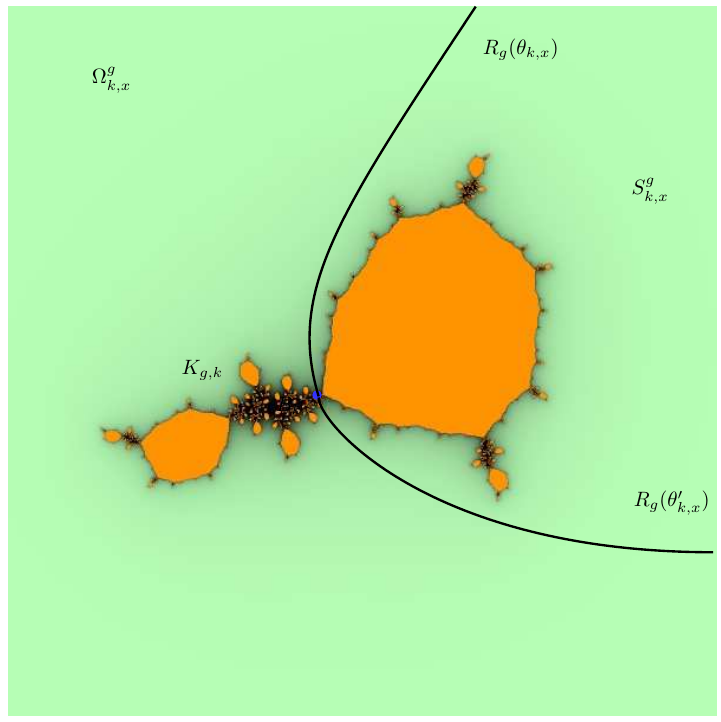}
    \caption{$K(g)$}
  \end{minipage}
 \end{center}
\end{figure}

   Let $g \in \comb(f)$ with no indifferent cycle, then $\L_\Q(g)=\L_\Q(f)$.
   By Claim \ref{rational}, for every $x \in \partial B_k$, there exists $(\theta_{k,x},\theta'_{k,x}) \in \L_\Q(g)$.
   Hence $\Omega_{k,x}^g$ be the connected of $\C \setminus (\overline{R_g(\theta_{k,x})} \cup \overline{R_g(\theta'_{k,x})})$ containing external rays $R_g(\theta)$ with $\theta \in E_{k,x} $. 
   The other connected component is denoted by $S_{k,x}^g:=\C \setminus \overline{ \Omega_{k,x}^g }$.
   Let 
   $$
        K_{g,k}:= \bigcap_{x \in \partial B_k^*} \overline{\Omega_{k,x}^g} \cap K(g).
   $$
   Then $K_{g,k}$ is a forward invariant compact set under $g^{p_k}$.

   \begin{clm}\label{filledin=}
      For every $k$, there are two distinct fixed points of $g^{p_k}$ in $K_{g,k}$.
   \end{clm}

   \begin{proof}
       Let $O$ be the smallest forward invariant set under $f^{p_k}$ which contains $x \in \partial B_k^*$ such that $L_{k,x} \cap \crit(f^{p_k})\neq \emptyset$, and $O':=f^{-p_k}(O)\cap \partial B_k$.
       By Theorem \ref{roeschyin}, we have $\partial B_k^*=\bigcup_{n \geq 0} f^{-np_k}(O) \cap \partial B_k$.
       Define two simply-connected domains
       $$U_k^g:=\bigcap_{y \in O'}\Omega^g_{k,y} \cap |\varphi_g|^{-1}((1,+\infty)), \quad V_k^g:=\bigcap_{y \in O}\Omega^g_{k,y} \cap |\varphi_g|^{-1}((d,+\infty)) .$$
       Then $g^{p_k}: U_k^g \to V_k^g$ is a proper map of degree $2$.
       Since every periodic in $J(g)$ is repelling, we can apply the well-known thicken technique \cite{milnor1992local} to the domains $U_k^g,V_k^g$ to obtain larger domains $\tilde{U}_g^k \Supset U_g^k ,\tilde{V}_g^k \Supset V_g^k$ such that $g^{p_k}: \tilde{U}_k^g \to \tilde{V}_k^g$ is a polynomial-like map of degree $2$.  
       Let $K_k(g)$ denote its filled-in Julia set.
       According to the property of the thicken technique, we have 
       $$K_k(g)=\{ z \in U_k^g: f^{np_k}(z) \in U_k^g \mbox{ for every $n$}\}.$$
       Then by the Straightening Theorem, $K_k(g)$ contains two fixed points of $g^{p_k}$ counted with multiplicity.
       Since $g$ has no indifferent cycle, $K_k(g)$ contains two distinct fixed points of $g^{p_k}$.

       To finish the proof of claim, it suffices to show that $K_{g,k}=K_k(g)$.
       Since $K_{g,k}\subset U_k^g$ is invariant under $g^{p_k}$, we have $K_{g,k}\subset K_k(g)$.
       Conversely, for every $z \notin K_{g,k}  $, there exists $x \in \partial B_k^*$ such that $z \in S_{k,x}^g$.
       Let $j$ be the smallest integer such that $f^{jp_k}(x)=x_j \in O$.
       Then $g^{jp_k}: S_{k,x}^g \to S_{k,x_j}^g$ is a conformal map.
       It follows that $g^{jp_k}(z) \notin \Omega_{k,x_j}^g$.
       Therefore $ g^{jp_k}(z) \notin V_k^g$ which implies that $z \notin K_k(g)$.
       This shows that $K_k(g)\subset K_{g,k}$.
       We conclude that $K_{g,k}=K_k(g)$.

       \end{proof}

       Let $E_k^0 \subset E_k$ be set of all $\theta$ such that $R_f(\theta)$ lands at the unique fixed point $\beta_k$ of $f^{p_k}$ on $\partial B_k$.
       By the local orientational-preserving property of $f^{p_k}$ near $\beta_k$, $f^{p_k}$ must fix the external rays of angles in $E_k^0$.
       It follows that $\tau_d^{p_k}$ fixes every angle in $E_k^0$. 
       Since $\L_\Q(f)=\L_\Q(g)$, then all external rays with angles in $E_k^0$ land at a common fixed point $\beta_k(g) \in K_{g,k}$.
       Moreover, every external ray with angle in $E_k \setminus E_k^0$ does not land at $\beta_k(g)$.
       By Claim \ref{filledin=}, there is another fixed point $\alpha_k(g) \in K_{g,k}\setminus \{ \beta_k(g) \}$.
       We claim that $\alpha_k(g)$ must be attracting.
       We prove it by contradiction.
       If $\alpha_k(g)$ is repelling, then there exists $\theta^* \in E_k$ such that $R_g(\theta^*)$ land at $\alpha_k(g)$.
       Suppose that $R_f(\theta^*)$ land at $x 
       \in \partial B_k$, then $x$ must be a periodic point of $f^{p_k}$ on $\partial B_k$ with some period $q$.
       If $q\geq 2$, then  $R_f(\tau_d^{(q-1)p_k}(\theta^*))$ does not land at $x$.
       Thus, $(\tau_d^{(q-1)p_k}(\theta^*),\theta^*)\notin \L_\Q(f)$.
       Notice that $\alpha_k$ is a fixed point of $g^{p_k}$.
       It follows that $(\tau_d^{(q-1)p_k}(\theta^*),\theta^*)\in \L_\Q(g)$.
       This contradicts the fact that $\L_\Q(f)=\L_\Q(g)$.
       Therefore, $x$ is the unique fixed point $\beta_k$ of $f^{p_k}$ on $\partial B_k$.
       This implies that $R_g(\theta^*)$ land $\beta_k(g)$, and $\alpha_k(g) =\beta_k(g)$.
       This contradicts Claim \ref{filledin=}.
       We conclude that $\alpha_k$ must be attracting.
       Combining the fact that $\alpha_k$ for $1\leq k\leq d-1$ have disjoint orbits, we see that $g$ is hyperbolic.
       By Fact \ref{pcf}, there exists a hyperbolic post-critically finite $g_0 \in \qc(g)$.
       Since $\L_\Q(g_0)=\L_\Q(g)=\L_\Q(f)$ and $f,g_0$ are both hyperbolic and post-critically finite, by Fact \ref{rigidity}, we have $f=g_0$.
       It follows that $g \in \qc(f)$.
       Hence $\comb(f)\subset \qc(f)$.
       This finishes the proof.

\end{proof}



\bibliographystyle{plain}
\bibliography{reference} 
\end{document}